\documentclass[12pt]{amsart}
\usepackage{amsfonts,amssymb,amsmath,amscd,amsthm, amstext,bm,color}
\usepackage[colorlinks, citecolor=blue,pagebackref,hypertexnames=false]{hyperref}
\usepackage{enumerate}
\usepackage{fancyhdr}
\usepackage{geometry}
\usepackage{graphicx}
\usepackage{epstopdf}
\usepackage{mathrsfs}
\usepackage{txfonts}
\usepackage{tikz}
\usetikzlibrary{arrows.meta, positioning, decorations.pathreplacing}
\usepackage{url}

\usepackage{comment}
\allowdisplaybreaks

\newcommand{\R}{\mathbb R}

\newcommand{\dist}{\operatorname{dist}}
\newcommand{\supp}{\operatorname{supp}}
\newcommand{\co}{\operatorname{co}}

\newcommand{\norm}[1]{\left\|#1\right\|}
\newcommand{\abs}[1]{\left|#1\right|}
\newcommand{\ip}[2]{\left\langle #1,#2\right\rangle}
\newcommand{\pv}{\operatorname{p.v.}}

\theoremstyle{plain}
\newtheorem{theorem}{Theorem}[section]
\newtheorem{proposition}[theorem]{Proposition}
\newtheorem{lemma}[theorem]{Lemma}

\theoremstyle{definition}
\newtheorem{definition}[theorem]{Definition}
\newtheorem{remark}[theorem]{Remark}

\numberwithin{equation}{section}

\title[Dunkl--BMO and generalized balayage of Carleson measures]{Dunkl--BMO functions representable in terms of generalized balayage of Carleson measures and an application}

\author{Qingdong Guo}
\address{Qingdong Guo, School of Mathematics and Statistics, Xiamen University of Technology, Xiamen 361024, P.R. China}
\email{qingdongmath@xmut.edu.cn}

\date{}

\begin{document}

\begin{abstract}
We characterize compactly supported functions $f$ in Dunkl--BMO spaces as those that can be written as 
\begin{equation*}
  f=g+\widetilde{S}_{\mu,\mathcal{P}},  
\end{equation*}
the sum of some bounded function $g$ and the generalized balayage $\widetilde{S}_{\mu,\mathcal{P}}$ of a Carleson measure on $\mathbb{R}^{N}\times(0,\infty)$ associated with the Dunkl--Poisson semigroup. As a direct application of this representation of Dunkl--BMO functions, we establish a new characterization of the Hardy space in the Dunkl setting via a principal value integral related to $\widetilde{S}_{\mu,\mathcal{P}}$. Moreover, analogous conclusions hold for the Dunkl--heat semigroup as well.
\end{abstract}


\maketitle

\tableofcontents

\section{Introduction}
The relationship among the space of functions of bounded mean oscillation, Carleson measures, and balayage has been a classical topic in modern harmonic analysis. The origins of this line of research can be traced back to the pioneering work of Carleson \cite[Theorem 2]{CarlesonLennart1976}. That is, for every function $f\in{\mathrm{BMO}(\mathbb{R}^{N})}$ (the bounded mean oscillation space) with compact support, there exist a sequence of bounded functions $\{b_{k}\}_{k=0}^{\infty}$ and a sequence of positive functions $\{t_{k}(\cdot)\}_{k\in{\mathbb{N}}}$ on $\mathbb{R}^{N}$ such that
\begin{equation*}
  f(x)=\sum_{k=1}^{\infty}\int_{\mathbb{R}^{N}}p_{t_{k}(y)}(x-y)b_{k}(y)dy+b_{0}(x),
\end{equation*}
where $p_{t}(x):=t^{-N}p(x/t)$ and $\{p_t\}_{t>0}$ denotes the classical Poisson semigroup, which is a particular example of the uniparametric families studied in \cite{CarlesonLennart1976}. 
Later, Uchiyama in \cite{Uchiyama1980} improved this result as
\begin{equation*}
  f(x)=\int_{\mathbb{R}^{N}}p_{t(y)}(x-y)b_{1}(y)dy+b_{0}(x),
\end{equation*}
and further showed that this representation can be written in the form (see also \cite{GarnettJones1982})
\begin{equation}\label{eq:bmo-decomposition}
  f(x)=\iint_{\mathbb{R}^{N+1}_{+}}p_{t}(x-y)d\mu(y,t)+b_{0}(x)=: S_{\mu,P}(x)+b_{0}(x),
\end{equation}
where $\mu$ is a Carleson measure on $\mathbb{R}^{N+1}_{+}:=\mathbb{R}^{N}\times(0,\infty)$, meaning that there exists $C>0$ such that for every cube $Q\subset\mathbb{R}^{N}$,
\begin{equation*}
  |\mu|(Q\times(0,\ell(Q)))\leq C|Q|,
\end{equation*}
where $\ell(Q)$ denotes the side length of $Q$, and $|\mu|$ denotes the total variation measure of $\mu$. And he further extended this result to Ahlfors-regular spaces under suitable conditions, using an iterative argument similar to that in \cite{CarlesonLennart1976}. 
Note that the constructions of decompositions analogous to \eqref{eq:bmo-decomposition} in the Euclidean setting and in Ahlfors-regular spaces, as presented in \cite{Uchiyama1980}, rely on the following key properties:

$\bullet$ the kernels appearing in these decompositions depend only on a single metric and satisfy certain conditions; 

$\bullet$ the measures of a ball are comparable to its radius, with comparison constants independent of its center.

By means of Green's theorem together with the Calder\'{o}n reproducing formula associated with the classical Poisson semigroup, Wilson \cite{Wilson1988} provided an alternative proof of the decomposition \eqref{eq:bmo-decomposition}. Subsequently, following Wilson's idea, the articles \cite{AlmeidaBetancor2019,Chen2017,GuoLy2025} on the balayage representations of BMO functions have been established in several nonclassical settings, including the Schrödinger setting and two Bessel-type settings. 

There is an observation concerning the balayage $S_{\mu,P}$. The decompositions of every compactly supported function 
$f\in\mathrm{BMO}(\mathbb{R}^{N})$ obtained in
\cite{CarlesonLennart1976,GarnettJones1982,Uchiyama1980}
take a form analogous to \eqref{eq:bmo-decomposition}, where $\mu$ is a Carleson measure constructed from $f$ and
\[
|S_{\mu,P}(x)|<\infty
\qquad\text{for\quad a.e.\,\,}x\in\mathbb{R}^{N}.
\]
Moreover, decompositions similar to \eqref{eq:bmo-decomposition}
have been obtained in \cite{AlmeidaBetancor2019,Chen2017,GuoLy2025,Wilson1988} for $\mathrm{BMO}$-type spaces associated with various operators with the resulting Carleson measures all being finite. Nevertheless, taking $\mu=\sum_{k=1}^{\infty}2^{kN}\delta_{(0,2^{k})}$, with $\delta_{(0,2^{k})}$ denoting the Dirac mass at $(0,2^{k})\in{\mathbb{R}^{N+1}_{+}}$, one can verify that $\mu$ is a Carleson measure on $\mathbb{R}^{N+1}_{+}$, yet the function $S_{\mu,P}(x)$ fails to be well-defined on $\mathbb{R}^{N}$. Hence, it is therefore natural to ask the following question:

\vspace{0.1cm}
\noindent
\textbf{Question:} \textit{Could we generalize the decomposition \eqref{eq:bmo-decomposition} by allowing the measure involved to be any Carleson measure?}
\vspace{0.1cm}

In this paper, we give an affirmative answer to this question by introducing a generalized balayage in the Dunkl setting that includes the classical Euclidean balayage as a special case. We then establish a characterization of Dunkl--$\mathrm{BMO}$ functions in terms of generalized balayages of Carleson measures, with an application to a Hardy space in the Dunkl setting.

\vspace{-0.3cm}
\subsection{Generalized balayage and main results}
To begin with, we provide some notation, with additional details provided in Section \ref{sec:dunkl-preliminaries}.
Let $R\subseteq \R^N\setminus\{0\}$ be a reduced finite root system, let $G$ be the finite reflection group generated by $R$, and let $\kappa\ge0$ be a
$G$-invariant multiplicity function on $R$.   The associated Dunkl measure is
$$
        d\omega(x):=\prod_{\alpha\in R}\abs{\ip{\alpha}{x}}^{\kappa(\alpha)}\,dx.
$$
Let $T_j:=T_{e_j}$, $j=1,\ldots,N$ be the Dunkl operators related to the standard basis $e_1,\ldots,e_N$.  The Dunkl Laplacian is $\Delta_D:=\sum_{j=1}^N T_j^2$.
The Dunkl--Poisson semigroup generated by $\sqrt{-\Delta_D}$ is defined by
$$
        \mathcal{P}_t f(x):=e^{-t\sqrt{-\Delta_D}}f(x):=\int_{\R^N}p_t(x,y)f(y)\,d\omega(y),\qquad t>0,\,x\in{\R^N},
$$
where
\begin{equation*}
        p_t(x,y)  :=  \int_{\R^N}   \frac{c_{N,\kappa}t}     {\bigl(t^2+\norm{x}^2+\norm{y}^2-2\ip{y}{\xi}\bigr)^{(\mathbf N+1)/2}}    \,d\mu_x(\xi),
\end{equation*}
and $\mu_x$ is a probability measure supported in $\co(O(x))$, the convex hull of the orbit of $x$ under $G$. The normalizing constant is
$$
        c_{N,\kappa}    :=2^{\mathbf N/2}\pi^{-1/2}\Gamma\bigl((\mathbf N+1)/2\bigr)/c_{\kappa},\quad c_{\kappa}:=\int_{\R^N}e^{-\norm{x}^{2}/2}\,d\omega(x).
$$

Let $\Re(\mathbb{R}^{N+1}_{+})$ be the set of all Carleson-type measures on $\mathbb{R}^{N+1}_{+}$. Consider the balayage of a measure $\mu\in\Re(\mathbb{R}^{N+1}_{+})$, defined as
$$
  S_{\mu,\mathcal{P}}(x):=\iint_{\mathbb{R}^{N+1}_{+}}p_{t}(x,y)d\mu(y,t).
$$
The notion of balayage can be understood as follows: if $\mu$ is a finite measure in $\Re(\mathbb{R}^{N+1}_{+})$, then in view of \eqref{eq:p-G-invariance}, we have $\|S_{\mu,\mathcal{P}}\|_{L^{1}(\mathbb{R}^{N}, d\omega)}\leq|\mu|(\mathbb{R}^{N+1}_{+})$, which implies $|S_{\mu,\mathcal{P}}(x)|<\infty$, for almost every $x\in{\mathbb{R}^{N}}$. However, when $\mu=\sum_{k=1}^{\infty}2^{k\mathbf{N}}\delta_{(0,2^{k})}$, we have $\mu\in{\Re(\mathbb{R}^{N+1}_{+})}$, whereas $S_{\mu,\mathcal{P}}(x)$ is not well-defined on $\mathbb{R}^{N}$. To extend the notion of balayage to all measures in $\Re(\mathbb{R}^{N+1}_{+})$, we first introduce the modified balayage $\widetilde{S}_{\mu,\mathcal{P}}$ by 
$$
  \widetilde{S}_{\mu,\mathcal{P}}(x):=\iint_{\mathbb{R}^{N+1}_{+}}p_{t}(x,y)-p_{t}(0,y)\cdot \mathbf{1}_{\{t>1\}}(y,t)d\mu(y,t),\qquad x\in{\mathbb{R}^{N}},
$$
where and in the sequel, $\mathbf{1}_{\{t>1\}}$ denotes the indicator of $\{(y,t)\in\mathbb{R}^{N+1}_{+}:\,\,t>1\}$.
It follows from \eqref{eq:local-integral} that, for any $\mu\in{\Re(\mathbb{R}^{N+1}_{+})}$, $\left|\widetilde{S}_{\mu,\mathcal{P}}(x)\right|<\infty$, for almost every  $x\in{\mathbb{R}^{N}}$.

Our first main result is as follows. The following notion of the Dunkl--BMO space
$\mathrm{BMO}(\mathbb{R}^{N},d\omega)$ is given in
Subsection~\ref{sec:dunkl-BMO-Hardy-Carleson}.

\begin{theorem}\label{thm:balayage-measure}
We have the following statements:
\begin{enumerate}
\item [$\rm(i)$] If $\mu\in{\Re(\mathbb{R}^{N+1}_{+})}$, then $\widetilde{S}_{\mu,\mathcal{P}}\in{\mathrm{BMO}(\mathbb{R}^{N}, d\omega)}$ and
    \begin{equation*}
      \left\|\widetilde{S}_{\mu,\mathcal{P}}\right\|_{\ast,\kappa}\lesssim\|\mu\|_{\mathcal{C}},
    \end{equation*}
     where the implicit constant is independent of $\mu$.

\item [$\rm(ii)$] Let $f\in{\mathrm{BMO}(\mathbb{R}^{N}, d\omega)}$ with compact support in $\mathbb{R}^{N}$. Then there exist $g\in{L^{\infty}(\mathbb{R}^{N}, d\omega)}$ and a Carleson measure $\mu\in{\Re(\mathbb{R}^{N+1}_{+})}$ such that
    \begin{equation}\label{eq:balayage-decomposition}
      f=g+\widetilde{S}_{\mu,\mathcal{P}},\qquad \mathrm{in}\quad\mathrm{BMO}(\mathbb{R}^{N}, d\omega),
    \end{equation}
    along with the bound
    \begin{equation*}
      \|g\|_{L^{\infty}(\mathbb{R}^{N}, d\omega)}+\|\mu\|_{\mathcal{C}}\lesssim\|f\|_{\ast,\kappa},
    \end{equation*}
where the implicit constant is independent of $f$. Here and in the sequel, the notation ``$\mathrm{in}$ $\mathrm{BMO}(\mathbb{R}^{N}, d\omega)$" means that equality is understood modulo constants.
\end{enumerate}
\end{theorem}  

\begin{remark}\label{rem-equivalent-norm}
$\rm(i)$ Theorem \ref{thm:balayage-measure} generalizes the results obtained in \cite{CarlesonLennart1976,GarnettJones1982,Uchiyama1980,Wilson1988}; when the multiplicity function $\kappa=0$, it reduces to the corresponding Euclidean result. 

$\rm(ii)$ By Theorem \ref{thm:balayage-measure}, the following equivalent
characterization of the $\mathrm{BMO}(\mathbb{R}^{N},d\omega)$ norm holds
for every compactly supported function $f$:
\begin{equation*}
  \|f\|_{\ast,\kappa}
  \sim
  \inf\left\{
    \|g\|_{L^{\infty}(\mathbb{R}^{N},d\omega)}
    +\|\mu\|_{\mathcal C}
  \right\},
\end{equation*}
where the infimum is taken over all representations of $f$ as in
\eqref{eq:balayage-decomposition}.
\end{remark}

As a consequence of Theorem \ref{thm:balayage-measure}, we derive the following novel characterization of the Hardy space $H^{1}(\mathbb{R}^{N}, d\omega)$ (see Subsection \ref{sec:dunkl-BMO-Hardy-Carleson}). For $t\in\R$ and $K>0$, define
\[
T_K(t)
=
\begin{cases}
-K, & t<-K,\\
t, &|t|\le K,\\
K, & t>K.
\end{cases}
\]
For $b\in{\mathrm{BMO}(\mathbb{R}^{N}, d\omega)}$ and $K>0$, set
\begin{equation}\label{eq:bN-notation}
b_K:=T_K\circ b=\max\{\min\{b,\,K\},\,-K\}.
\end{equation}
For $f\in L^1(\mathbb{R}^{N}, d\omega)$ and $b\in{\mathrm{BMO}(\mathbb{R}^{N}, d\omega)}$, if the limit
\begin{equation*}
  \lim_{K\to\infty}
\int_{\mathbb{R}^{N}}f(x)b_K(x)\,d\omega(x)
\end{equation*}
exists and is finite, then we say that the principal value integral
of $fb$ exists, and define
\begin{equation*}
\pv\int_{\mathbb{R}^{N}} f(x)b(x)d\omega(x)
:=
\lim_{K\to\infty}
\int_{\mathbb{R}^{N}}f(x)b_K(x)\,d\omega(x).
\end{equation*}
Note that if $fb\in{L^{1}(\mathbb{R}^{N}, d\omega)}$, then this principal value integral coincides with the usual integral by the dominated convergence theorem. We emphasize that the principal value integral of $fb$ is defined by truncating the values of $b$, rather than by spatial truncation.

\begin{theorem}\label{thm-Hardy-balayage}
The following statements are mutually equivalent:
\begin{enumerate}
\item [$\rm(i)$]
$f\in{H^{1}(\mathbb{R}^{N}, d\omega)}$;

\item [$\rm(ii)$]
$f\in{L^{1}(\mathbb{R}^{N}, d\omega)}$, the principal value integral 
$\operatorname{p.v.}\int_{\mathbb R^{N}} f(y)\widetilde{S}_{\mu,P}(y)d\omega(y)$ exists for every $\mu\in \mathfrak R(\mathbb R^{N+1}_{+})$, and 
\begin{equation*}
  \sup_{\substack{\mu\in \mathfrak R(\mathbb R^{N+1}_{+})\\ \|\mu\|_{\mathcal C}=1,\ c\in \mathbb{R}}}
\left|
\operatorname{p.v.}\int_{\mathbb R^{N}} f(y)\,\left(\widetilde{S}_{\mu,P}(y)+c\right)\,d\omega(y)
\right|<\infty.
\end{equation*}
\end{enumerate}
Moreover, we have
\begin{equation*}
  \|f\|_{H^{1}(\mathbb{R}^{N}, d\omega)}\thicksim\|f\|_{L^{1}(\mathbb{R}^{N}, d\omega)}+\sup\limits_{\substack{\mu\in{\Re(\mathbb{R}^{N+1}_{+})},\\\|\mu\|_{\mathcal{C}}=1}}\left|\operatorname{p.v.}\int_{\mathbb{R}^{N}}f(y)\cdot \widetilde{S}_{\mu,\mathcal{P}}(y)d\omega(y)\right|.
\end{equation*}
\end{theorem}

\subsection{Outline of the main techniques} 
To prove Theorem \ref{thm:balayage-measure}, we make some observations: $(1)$ The Green formula related to Dunkl operators holds for the bounded, regular and G-invariant domain in $\mathbb{R}^{N}\times(0,\infty)$. Hence, Wilson's idea \cite{Wilson1988} is not feasible in the Dunkl setting. $(2)$ The Dunkl--Poisson integral involves Dunkl and Euclidean metrics which are not equivalent, and the Dunkl measure of ball relies on center and radius. 

This paper will develop the methods in \cite{CarlesonLennart1976,Uchiyama1980} via some new techniques to bypass the above difficulties. To be more precise,

$\bullet$ we introduce an auxiliary kernel 
\begin{equation*} 
  K_{t}(x,y)=\frac{p_{t}(x,y)\chi_{t}^{\mathcal{A}}(x,y)}  {\phi_{t}^{\mathcal{A}}(x)},\qquad x,y\in{\mathbb{R}^{N}},t>0.
\end{equation*}
where $\chi_{t}^{\mathcal{A}}(x,y)$ and $\phi_{t}^{\mathcal{A}}(x)$ are auxiliary functions; see Subsection \ref{sub:approximating-kernel}. The key step in constructing this kernel is to apply a suitable smooth cutoff to the Poisson kernel so that the resulting kernel is supported in a Euclidean ball. It is designed to overcome the difficulties caused by the orbit structure and is indispensable for constructing the bounded functions with pairwise disjoint supports in Proposition \ref{pro:function-bn} below. Furthermore, it provides the essential connection underlying the decomposition with respect to $\mathbf{p}_{t}(x,y)$ (see Theorem \ref{thm-BMO-onesetp}), and it serves as an effective bridge between the Dunkl metric and the Euclidean metric (see Lemmas \ref{lem:lowbounded-kernel} and \ref{lem:cube-estimate}).
 
$\bullet$ To prove the key estimate $\|\psi^{(1)}\|_{\ast,\kappa}<\frac{1}{4}$ (see Lemma \ref{lem:psi-estimate}), the main technical difficulties associated with the Dunkl measure and metric can be circumvented by suitably decomposing $K_t(x,y)-p_t(x,y)$.

The article is arranged as follows. In Section \ref{sec:dunkl-preliminaries}, we recall some basic facts and introduce the notions needed in the Dunkl setting. In Section \ref{sec:convolution-decomposition}, we establish a decomposition of Dunkl--BMO functions via a modified Dunkl--Poisson integral. In Section \ref{sec:thm1.1-1.2}, we use this decomposition to prove Theorem \ref{thm:balayage-measure} and also prove Theorem \ref{thm-Hardy-balayage}. Lastly, Section \ref{sec:Dunkl-heat} is dedicated to establishing a decomposition of compactly supported Dunkl--BMO functions in terms of the generalized balayage of Carleson measures associated with the Dunkl--heat semigroup.

In the end, we make some conventions on notation. In this paper, let $\mathbb{N}:=\{1,2,\ldots\}$ and $\mathbb{Z}:=\{0,\pm1,\pm2,\ldots\}$. Let
\[
B(x,r):=\left\{y\in{\mathbb{R}^{N}}:\,\|x-y\|<r\right\}
\]
stand for the ball in $\mathbb{R}^{N}$ with center $x$ and radius $r$. Denote the cube in $\mathbb{R}^{N}$ by 
\begin{equation*}
  Q:=Q(x,\ell(Q)):=\left\{y\in{\mathbb{R}^{N}}:\,-\ell(Q)/2\leq y_{i}-x_{i}<\ell(Q)/2,\,\,i=1,2,\ldots,N\right\},
\end{equation*}
where $x,\ell(Q)$ are center and side length of $Q$, respectively.
Moreover, dilated cubes are taken with the corresponding half-open convention.
The notation $f\lesssim g$ and $f\thicksim g$ mean that there exists $C\ge1$ such that $f\le Cg$ and $f/C\le g\le Cf$, respectively; 
subscripts on $\lesssim$ indicate the allowed structural dependence of the implicit constant.

\section{Dunkl preliminaries}\label{sec:dunkl-preliminaries}
In this section, we review some preliminaries of Dunkl theory, including basic definitions and notation, Dunkl–BMO and Dunkl–Hardy spaces, Carleson-type measures, and some useful properties of the Dunkl–Poisson kernel, which are presented in Subsections \ref{sec:dunkl-theory}, \ref{sec:dunkl-BMO-Hardy-Carleson} and \ref{sec:dunkl-poisson}, respectively.

\subsection{Dunkl theory}\label{sec:dunkl-theory}

For background on Dunkl operators and related applications, we refer the reader to
\cite{AmriHammi2019,AmriSifi2012,deJeu1993,Dunkl1989,Dunkl1991,Opdam1993}.  
The Euclidean inner product and norm on $\R^N$ are denoted by $\ip{x}{y}$ and $\norm{x}:=(\sum_{i=1}^{N}|x_{i}|^{2})^{1/2}$.

Let $R\subseteq \R^N\setminus\{0\}$ be a reduced root system.  Thus
$$
        R\cap \R\alpha=\{\alpha,-\alpha\}, \qquad \sigma_\alpha(R)=R,
$$
where
$$
        \sigma_\alpha(x):=x-2\frac{\ip{x}{\alpha}}{\norm{\alpha}^2}\alpha.
$$
We normalize the root normals by $\norm{\alpha}^2=2$.  
This choice does not change the reflection or the Dunkl difference quotient, and changes $d\omega$ only by an overall constant.

Let $G$ be the finite reflection group generated by  $\{\sigma_\alpha:\alpha\in R\}$.  
If $E\subseteq \R^N$ and $x\in\R^N$, set
$$
        O(E):=\bigcup_{\sigma\in G}\sigma(E),  \qquad   O(x):=O(\{x\}).
$$
The Dunkl metric 
\begin{equation*}
  d(x,y):=\min\limits_{\sigma\in{G}}\|x-\sigma(y)\|.
\end{equation*}
satisfies $\abs{\norm{x}-\norm{y}}\le d(x,y)\le \norm{x-y}$, $x,y\in\R^N$. Moreover, we have the following elementary inequality:
\begin{equation*} 
        d(x_0,y_0)=\min_{\sigma\in G}\norm{x_0-\sigma (y_0)}  \le   \sqrt{\norm{x_0}^2+\norm{y_0}^2-2\ip{y_0}{\xi}}\leq \max_{\sigma\in G}\norm{x_0-\sigma (y_0)},    \qquad \xi\in\co(O(x_0)).
\end{equation*}

Let $\kappa:R\to[0,\infty)$ be $G$-invariant.  The Dunkl operator in direction $\xi\in\R^N$ is
$$
     T_\xi f(x):=\partial_\xi f(x)   +\sum_{\alpha\in R}
        \frac{\kappa(\alpha)}{2}\ip{\alpha}{\xi}   \frac{f(x)-f(\sigma_\alpha x)}{\ip{\alpha}{x}}.
$$
The Dunkl measure is $d\omega(x)=\omega(x)dx$, where
$$
     \omega(x):=\prod_{\alpha\in R}\abs{\ip{\alpha}{x}}^{\kappa(\alpha)}.
$$
It is $G$-invariant.  Put
$$
        \mathbf N:=N+\sum_{\alpha\in R}\kappa(\alpha).
$$
The measure $d\omega$ satisfies the doubling condition in the sense that
$$
     \lambda^N\omega(B(x,r)) \lesssim  \omega(B(x,\lambda r))  \lesssim
        \lambda^{\mathbf N}\omega(B(x,r)),   \qquad \lambda\ge1,
$$
and has upper and lower bounded estimates of any orbit of ball,
$$
     \omega(B(x,r))   \le \omega(O(B(x,r)))    \le   |G|\,\omega(B(x,r)).
$$
where $x\in\mathbb{R}^{N}$ and $r>0$.

\subsection{Dunkl--BMO, Dunkl--Hardy spaces and Carleson-type measure}\label{sec:dunkl-BMO-Hardy-Carleson}

The definitions of $\mathrm{BMO}(\mathbb{R}^{N}$ $,d\omega)$, $\Re(\mathbb{R}^{N+1}_{+})$ and $H^{1}(\mathbb{R}^{N},d\omega)$ are as follows.
\begin{definition}
A real-valued function $f\in{L_{\mathrm{loc}}^{1}(\mathbb{R}^{N}, d\omega)}$ is said to be in $\mathrm{BMO}(\mathbb{R}^{N}, d\omega)$, if
\begin{equation*}
\|f\|_{\ast,\kappa}:=\sup\limits_{x\in{\mathbb{R}^{N}},\,r>0}\frac{1}{\omega(B(x,r))}\int_{B(x,r)}\left|f(y)-\mathcal{M}_{B(x,r)}(f)\right|d\omega(y)<\infty,
\end{equation*}
where 
\begin{equation*}
 \mathcal{M}_{B(x,r)}(f):=\frac{1}{\omega(B(x,r))}\int_{B(x,r)}f(y)d\omega(y).
\end{equation*}
Here and in the sequel, $B$ under the supremum always represents a ball in $\mathbb{R}^{N}$.
\end{definition}
\begin{definition}
A real-valued signed Borel measure $\mu$ on $\mathbb{R}^{N+1}_{+}$ is in $\Re(\mathbb{R}^{N+1}_{+})$ when there exists $C>0$ such that for any ball $B\subset\mathbb{R}^{N}$,
\begin{equation*}
  |\mu|(\widehat{B})\leq C\omega(B), 
\end{equation*}
where $|\mu|$ represents the total variation measure of $\mu$, and $\widehat{B}:=B\times(0, r)$ with $r$ the radius of $B$. We equip $\Re(\mathbb{R}^{N+1}_{+})$ with the norm 
\begin{equation*}
\|\mu\|_{\mathcal{C}}:=\sup_{B\subset\mathbb{R}^{N}}\frac{|\mu|(\widehat{B})}{\omega(B)}.
\end{equation*}
Such measures in $\Re(\mathbb{R}^{N+1}_{+})$ are referred to as Carleson-type measures.
\end{definition}

We denote by $H^{1}(\mathbb{R}^{N}, d\omega)$ the Hardy space associated with the Dunkl operator defined in \cite{CS1977}.
\begin{definition}
A measurable function $a$ is a $(1,\infty)$-atom if there exists a ball $B\subset\mathbb{R}^{N}$ such that $\supp a\subset B$,
\begin{equation*}
\|a\|_{L^{\infty}(\mathbb{R}^{N}, d\omega)}\leq \frac{1}{\omega(B)}\quad \mathrm{and}\quad \int_{\mathbb{R}^{N}}a(x)d\omega(x)=0.
\end{equation*}
A function $f\in{L^{1}}(\mathbb{R}^{N}, d\omega)$ is in $H^{1}(\mathbb{R}^{N}, d\omega)$ if and only if $f=\sum_{j=1}^{\infty}\alpha_{j}a_{j}$ in $L^{1}(\mathbb{R}^{N}, d\omega)$, where for every $j$, $a_{j}$ is a $(1,\infty)$-atom and $\alpha_{j}\in{\mathbb{C}}$ with $\sum_{j=1}^{\infty}|\alpha_{j}|<\infty$. The space $H^{1}(\mathbb{R}^{N}, d\omega)$ is a Banach space endowed with the norm
\begin{equation*}
\|f\|_{H^{1}(\mathbb{R}^{N}, d\omega)}:=\inf\left\{\sum_{j=1}^{\infty}|\alpha_{j}|\right\},
\end{equation*}
where the infimum is taken over all representations of $f$ as above.
\end{definition}
 
\subsection{Properties of Dunkl--Poisson kernel}\label{sec:dunkl-poisson}

The Dunkl–Poisson integral can be expressed in the form of a Dunkl--convolution.
The following standard positivity, symmetry, and pointwise estimates for the Dunkl--Poisson kernel $p_t(x,y)$ follow from the positivity of the Dunkl intertwining operator and the  kernel estimates for the Dunkl--convolution and auxiliary analytical tools.

For $x,y\in\R^N$, $t>0$,
\begin{equation}\label{eq:p-symmetry}
        p_t(x,y)=p_t(y,x)>0, 
\end{equation}
and
\begin{equation}\label{eq:p-G-invariance}
\int_{\mathbb{R}^{N}}p_{t}(x,y)d\omega(y)=1;
\end{equation}
see \cite{Margit2003,MargitMichael1998,ThangaveluXu2005}.
There exist positive constants $C_{0},C_{1},C_{2},C_{3}$ depending solely on the dimension $\mathbf N$ such that the following mixed-metric upper and lower estimates hold:
\begin{align}\label{eq:poisson-upper}
\notag\frac{t}{V(x,y,t+\|x-y\|)}\frac{C_{0}}{t+\norm{x-y}}
&\leq p_t(x,y)\\
&\leq C_{1}\min\left\{\frac{1}{V(x,y, t+d(x,y))}\frac{t}{t+\|x-y\|},\frac{t}{\big(t+\big|\|x\|-\|y\|\big|\big)^{\mathbf{N}+1}}\right\},
\end{align}
and
\begin{equation}\label{eq:poisson-smooth}
\left\|\nabla_{x}p_{t}(x,y)\right\|\lesssim\frac{p_{t}(x,y)}{t+d(x,y)}\leq C_{2}\frac{1}{V(x,y, t+d(x,y))}\frac{t}{t+\|x-y\|}\frac{1}{t+d(x,y)};
\end{equation}
see \cite{JacekHejna2020,JiuLi2023}. Moreover, if $\|y-y_{1}\|\leq (t+d(x,y))/2$, then
\begin{equation}\label{eq:poisson-Holder}
\left|p_{t}(x,y)-p_{t}(x,y_{1})\right|\leq C_{3}\frac{\|y-y_{1}\|}{V(x,y, t+d(x,y))}\frac{1}{t+\|x-y\|}\frac{t}{t+d(x,y)},
\end{equation}
where $x,y,y_{1}\in\R^N$, $t>0$. Here, we denote $\nabla_{x}:=(\partial_{x_{1}},\ldots,\partial_{x_{N}})$ and $V(x,y, t):=\max\{\omega(B(x,t)), \omega(B(y,t))\}$. The estimate \eqref{eq:poisson-Holder} follows directly from \eqref{eq:poisson-smooth}; we omit the details.

\section{A BMO decomposition in terms of a modified Dunkl--Poisson integral}\label{sec:convolution-decomposition}
In this section, we establish the following decomposition theorem for
compactly supported functions in $\mathrm{BMO}(\mathbb{R}^{N},d\omega)$, which serves as a key intermediate step toward proving Theorem~\ref{thm:balayage-measure}.
\begin{theorem}\label{thm-BMO-onesetp}
Let $f\in{\mathrm{BMO}(\mathbb{R}^{N}, d\omega)}$ with compact support and $\|f\|_{\ast,\kappa}=1$. Then there exist bounded functions $\{g_{i}\}_{i=0}^{\infty}$ and decreasing positive numbers $\{r_{i}\}_{i\in{\mathbb{N}}}$ such that 
\begin{equation*} 
 f=\sum\limits_{i=1}^{\infty}\int_{\R^N}\mathbf{p}_{r_{i}}(\cdot,y)g_{i}(y)\,d\omega(y)+g_{0}\qquad \mathrm{in}\quad\mathrm{BMO}(\mathbb{R}^{N}, d\omega),
\end{equation*}
with 
\begin{equation*}
  \left\|\sum\limits_{i=1}^{\infty}|g_{i}|\right\|_{L^{\infty}(\mathbb{R}^{N}, d\omega)}\lesssim1,\quad \mathrm{and} \quad \|g_{0}\|_{L^{\infty}(\mathbb{R}^{N}, d\omega)}\lesssim1,
\end{equation*}
where the implicit constants are independent of $f,g_{i},g_{0}$. Here and throughout the sequel, we use the notation
\begin{equation*}
\mathbf{p}_{r_{i}}(\cdot,y):=p_{r_{i}}(\cdot,y)-p_{r_{i}}(0,y)\cdot \mathbf{1}_{\{r_{i}>1\}}(y,r_{i}).
\end{equation*}
\end{theorem}

To show Theorem \ref{thm-BMO-onesetp}, we first construct an auxiliary kernel as a replacement for the Poisson kernel $p_t(x,y)$, and establish its properties in Subsection \ref{sub:approximating-kernel}. Subsequently, by using the properties of this auxiliary kernel, we provide the proof of Theorem \ref{thm-BMO-onesetp} in Subsection \ref{sub:a decomposition}.

\subsection{An auxiliary kernel}\label{sub:approximating-kernel}
\hskip\parindent

Let $\eta\in{C^{\infty}([0,\infty))}$ satisfy $0\leq\eta(y)\leq1$, $y\in[0,\infty)$, and $\eta(y)=1$, $y\in{[0,1]}$, and $\eta(y)=0$, $y\in{[2,\infty)}$. Fix a constant \(\mathcal{A}>0\). For $x,y\in{\mathbb{R}^{N}}$ and $t>0$, define 
\begin{equation*}
  \chi_{t}^{\mathcal{A}}(x,y):=\eta\left(\frac{\|x-y\|}  {\mathcal{A}^{4}t}\right), 
\end{equation*}
and set
\begin{equation*}
  \phi_{t}^{\mathcal{A}}(x):=\int_{\mathbb{R}^{N}}p_{t}(x,y)\chi_{t}^{\mathcal{A}}(x,y)d\omega(y).
\end{equation*}
We then define the auxiliary kernel
\begin{equation}\label{eq:approximating-kernel}
  K_{t}(x,y):=\frac{p_{t}(x,y)\chi_{t}^{\mathcal{A}}(x,y)}  {\phi_{t}^{\mathcal{A}}(x)},\quad x,y\in{\mathbb{R}^{N}},t>0.
\end{equation}

We give the following properties for the above functions.
\begin{lemma}\label{lem-kernel-properties}
Fix a constant $\mathcal{A}$ such that $\mathcal{A}>\sqrt{N}$. The following statements hold:

$\rm{(a1)}$ For any $x,y,z\in{\mathbb{R}^{N}}$ and $t>0$,
\begin{equation*}
  \left|\chi_{t}^{\mathcal{A}}(x,y)-\chi_{t}^{\mathcal{A}}(z,y)\right|\lesssim\frac{\|x-z\|}{\mathcal{A}^{4}t},
\end{equation*}
where the implicit constant only depends on $\eta$.

$\rm{(a2)}$ There exists a constant $C_{4}>0$ only depending on $\mathbf{N}$ such that for any $x,z\in{\mathbb{R}^{N}}$ and $t>0$,
\begin{eqnarray*} 
C_{4}\leq\phi_{t}^{\mathcal{A}}(x)\leq1,
\end{eqnarray*}
and
\begin{eqnarray*} 
\left|\phi_{t}^{\mathcal{A}}(x)-\phi_{t}^{\mathcal{A}}(z)\right|\lesssim(1+C_{2})\frac{\|x-z\|}{\mathcal{A}^{4}t}.
\end{eqnarray*}

$\rm{(a3)}$ Let $x,y,z\in{\mathbb{R}^{N}}$ and $t>0$. Then $\supp K_{t}(x,\cdot)\subset B(x,2\mathcal{A}^{4}t)$ and
\begin{equation*}
  \int_{\mathbb{R}^{N}}K_{t}(x,y)d\omega(y)=1,\qquad K_{t}(x,y)\geq0,
\end{equation*}
and
\begin{equation*}
p_t(x,y)\cdot\mathbf{1}_{\{\|x-y\|\leq\mathcal{A}^{4}t\}}\leq K_{t}(x,y)\leq p_t(x,y)\cdot\mathbf{1}_{\{\|x-y\|\leq2\mathcal{A}^{4}t\}}/C_{4},
\end{equation*}
and if $\|x-z\|\leq (t+d(x,y))/2$, then 
\begin{equation*}
\left|K_{t}(x,y)-K_{t}(z,y)\right|\lesssim_{C_{1},C_{2},C_{3},C_{4}}\frac{1}{V(x,y, t+d(x,y))}\frac{t}{t+\|x-y\|}\left(\frac{\|x-z\|}{t+d(x,y)}+\frac{\|x-z\|}{\mathcal{A}^{4}t}\right).
\end{equation*}
Moreover,  
\begin{equation*}
\left\|\nabla_{x}K_{t}(x,y)\right\|\lesssim_{C_{1},C_{2},C_{4}}\frac{1}{V(x,y, t+d(x,y))}\frac{t}{t+\|x-y\|}\left(\frac{1}{t}+\frac{1}{t+d(x,y)}\right).
\end{equation*}

$\rm{(a4)}$ Let $x_{0}\in{\mathbb{R}^{N}}$, $t>0$ and $x\in{Q(x_{0},t)}$. Then there exists a constant $C_{5}\geq1$ depending only on $\mathbf{N}$ such that
\begin{equation*}
  \int_{Q(x_{0},t)}K_{t}(x,y)d\omega(y)\geq\frac{4}{9C_{5}}.
\end{equation*} 

\end{lemma}
\begin{proof}   
$\rm{(a1)}$ holds due to the smoothness of $\eta$. We now consider $\rm{(a2)}$. Firstly, the upper bound of $\phi_{t}^{\mathcal{A}}$ follows by \eqref{eq:p-G-invariance}. Using \eqref{eq:poisson-upper}, for any $x\in{\mathbb{R}^{N}}$,
\begin{align*}
  \int_{\mathbb{R}^{N}}p_{t}(x,y)\chi_{t}^{\mathcal{A}}(x,y)d\omega(y)
  &\geq C_{0}\int_{B(x,t)}\frac{1}{\omega(B(x, t+\|x-y\|))}\frac{t}{t+\|x-y\|}d\omega(y)\\
  &\geq C_{0}\int_{B(x,t)}\frac{1}{\omega(B(x, t+t))}\frac{t}{t+t}d\omega(y)\\
  &\gtrsim \frac{C_{0}}{2^{\mathbf{N}+1}}\\
  &=: C_{4}.
\end{align*}
 
Suppose first that $x,z\in{\mathbb{R}^{N}}$ and $t>0$ satisfy $\|x-z\|\geq t\mathcal{A}^{4}/2$. Then
\begin{eqnarray*} 
\left|\phi_{t}^{\mathcal{A}}(x)-\phi_{t}^{\mathcal{A}}(z)\right|\leq2\leq\frac{4\|x-z\|}{\mathcal{A}^{4}t}.
\end{eqnarray*}
It remains to consider the case in which $x,z\in{\mathbb{R}^{N}}$, $t>0$ and $\|x-z\|< t\mathcal{A}^{4}/2$. Write
\begin{align*}
  \left|\phi_{t}^{\mathcal{A}}(x)-\phi_{t}^{\mathcal{A}}(z)\right|
   &=\left|\phi_{t}^{\mathcal{A}}(x)-1+1-\phi_{t}^{\mathcal{A}}(z)\right|\\
   &\leq\int_{\mathbb{R}^{N}}\left|p_{t}(x,y)-p_{t}(z,y)\right|\left(1-\chi_{t}^{\mathcal{A}}(x,y)\right)d\omega(y)+\int_{\mathbb{R}^{N}}p_{t}(z,y)\left|\chi_{t}^{\mathcal{A}}(x,y)-\chi_{t}^{\mathcal{A}}(z,y)\right|d\omega(y)\\
   &=: I+II.
\end{align*}
Applying \eqref{eq:p-symmetry} and \eqref{eq:poisson-smooth}, 
\begin{align*}
  I
  &\leq C_{2}\int_{0}^{1}\int_{\|x-y\|\geq t\mathcal{A}^{4}}\frac{\|x-z\|}{\omega(B(\xi_{s},t+d(\xi_{s},y)))}\frac{1}{t+\|\xi_{s}-y\|}\frac{t}{t+d(\xi_{s},y)}\left(1-\chi_{t}^{\mathcal{A}}(x,y)\right)d\omega(y)ds\lesssim C_{2}\frac{\|x-z\|}{\mathcal{A}^{4}t},
\end{align*}
where $\xi_{s}=sx+(1-s)z$. In the last inequality, we apply the fact that 
\begin{align}\label{eq:sequence-Dunkl}
  &\int_{\mathbb{R}^{N}}\frac{1}{\omega(B(\xi_{s},t+d(\xi_{s},y)))}\frac{t}{t+d(\xi_{s},y)}d\omega(y)\\\notag
  &\quad\leq\left(\int_{d(\xi_{s},y)\leq t}+\sum\limits_{i=0}^{\infty}\int_{d(\xi_{s},y)\thicksim2^{i}t}\right)\frac{1}{\omega(B(\xi_{s},t+d(\xi_{s},y)))}\frac{t}{t+d(\xi_{s},y)}d\omega(y)\\\notag
  &\quad\lesssim1,\notag
\end{align}
where the implicit constant is independent of $s,t$. Moreover, by $\rm{(a1)}$,
\begin{align*}
  II&\lesssim\frac{\|x-z\|}{\mathcal{A}^{4}t}\int_{\mathbb{R}^{N}}p_{t}(z,y)d\omega(y)\lesssim\frac{\|x-z\|}{\mathcal{A}^{4}t}.
\end{align*}

Consider $\rm{(a3)}$. Let  $x,y,z\in{\mathbb{R}^{N}}$ and $t>0$ satisfy $\|x-z\|\leq (t+d(x,y))/2$. By $\rm{(a2)}$, \eqref{eq:poisson-upper} and \eqref{eq:poisson-Holder},   
\begin{align*}
  \left|K_{t}(x,y)-K_{t}(z,y)\right|
  &\leq\chi_{t}^{\mathcal{A}}(x,y)\left|\frac{p_{t}(x,y)-p_{t}(z,y)}  {\phi_{t}^{\mathcal{A}}(x)}\right|+p_{t}(z,y)\left|\frac{\chi_{t}^{\mathcal{A}}(x,y)-\chi_{t}^{\mathcal{A}}(z,y)}  {\phi_{t}^{\mathcal{A}}(x)}\right|\\
  &\quad+p_{t}(z,y)\chi_{t}^{\mathcal{A}}(z,y)\left|\frac{1}{\phi_{t}^{\mathcal{A}}(x)}-\frac{1}{\phi_{t}^{\mathcal{A}}(z)}\right|\\
  &\lesssim\left(\frac{C_{3}}{C_{4}}+\frac{C_{1}}{C_{4}}+\frac{C_{1}(1+C_{2})}{C_{4}^{2}}\right)\frac{1}{V(x,y, t+d(x,y))}\frac{t}{t+\|x-y\|}\left(\frac{\|x-z\|}{t+d(x,y)}+\frac{\|x-z\|}{\mathcal{A}^{4}t}\right). 
\end{align*}
 
Let $x,y\in{\mathbb{R}^{N}}$, $t>0$ and $i\in\{1,2,\ldots,N\}$. We have
\begin{equation*}
\left|\partial_{x_{i}}K_{t}(x,y)\right|\leq\frac{\left|\partial_{x_{i}}[p_{t}(x,y)\chi_{t}^{\mathcal{A}}(x,y)]\cdot\phi_{t}^{\mathcal{A}}(x)\right|}{\left(\phi_{t}^{\mathcal{A}}(x)\right)^{2}}
+\frac{\left|[p_{t}(x,y)\chi_{t}^{\mathcal{A}}(x,y)]\cdot\partial_{x_{i}}\phi_{t}^{\mathcal{A}}(x)\right|}{\left(\phi_{t}^{\mathcal{A}}(x)\right)^{2}}.
\end{equation*}
From \eqref{eq:poisson-smooth}, 
\begin{align*}
\left|\partial_{x_{i}}[p_{t}(x,y)\chi_{t}^{\mathcal{A}}(x,y)]\cdot\phi_{t}^{\mathcal{A}}(x)\right|
&\leq\left|\partial_{x_{i}}[p_{t}(x,y)\chi_{t}^{\mathcal{A}}(x,y)]\right|\\
&\lesssim\left(C_{1}+C_{2}\right)\frac{1}{\omega(B(x,t+d(x,y)))}\frac{t}{t+\|x-y\|}\left(\frac{1}{t+d(x,y)}+\frac{1}{\mathcal{A}^{4}t}\right).
\end{align*}
It then follows from \eqref{eq:sequence-Dunkl} and $d(x,y)\leq\|x-y\|$ that 
\begin{equation*}
\int_{\mathbb{R}^{N}}\left|t\partial_{x_{i}}[p_{t}(x,y)\chi_{t}^{\mathcal{A}}(x,y)]\right|d\omega(y)\lesssim\int_{\mathbb{R}^{N}}\frac{1}{\omega(B(x,t+d(x,y)))}\frac{t}{t+\|x-y\|}d\omega(y)\lesssim1.
\end{equation*}
Then by the dominated convergence theorem, 
\begin{equation*}
\partial_{x_{i}}\phi_{t}^{\mathcal{A}}(x)=\int_{\mathbb{R}^{N}}\partial_{x_{i}}[p_{t}(x,y)\chi_{t}^{\mathcal{A}}(x,y)]d\omega(y).
\end{equation*}
Combining these facts, we obtain
\begin{equation*}
\left\|\nabla_{x}K_{t}(x,y)\right\|\lesssim\left(\frac{C_{1}}{C_{4}}+\frac{C_{2}}{C_{4}^{2}}\right)\frac{1}{V(x,y, t+d(x,y))}\frac{t}{t+\|x-y\|}\left(\frac{1}{t}+\frac{1}{t+d(x,y)}\right).
\end{equation*}

We now prove $\rm{(a4)}$. Since $\mathcal{A}>\sqrt{N}$, by $\rm{(a3)}$ and \eqref{eq:poisson-upper}, there exists a constant $M>0$ such that
\begin{align*}
\int_{Q(x_{0},t)}K_{t}(x,y)d\omega(y)
&\geq\int_{Q(x_{0},t)}p_{t}(x,y)d\omega(y)\\
&\geq C_{0}\int_{B(x_{0},t/2)}\frac{t}{\omega(B(x,t+\|x-y\|))}\frac{1}{t+\norm{x-y}}d\omega(y)\\
&\gtrsim\frac{C_{0}}{3^{\mathbf{N}+1}}\geq\frac{4}{9M}.
\end{align*}
If $M<1$, set $C_{5}:=M+1$. Otherwise, set $C_{5}:=M$. Therefore, the proof of this lemma is complete.
\end{proof}

\subsection{Proof of Theorem \ref{thm-BMO-onesetp}}\label{sub:a decomposition}
\hskip\parindent

In this subsection, we give the proof of Theorem \ref{thm-BMO-onesetp}. We begin with some basic properties of $\mathrm{BMO}(\mathbb{R}^{N}, d\omega)$ that will be used below.
\begin{lemma}\label{lem-bmo-properties}
Let $f\in{\mathrm{BMO}(\mathbb{R}^{N}, d\omega)}$. Then for any $x,y\in{\mathbb{R}^{N}}$, $r,r_{1}>0$ with $r<r_{1}$, we have that

$\rm{(b1)}$ 
\begin{equation*}
  \left|\mathcal{M}_{B(x,r)}(f)-\mathcal{M}_{B(x,r_{1})}(f)\right|\lesssim\left(1+\log_{2}(r_{1}/r)\right)\|f\|_{\ast,\kappa};
\end{equation*}

$\rm{(b2)}$ 
\begin{eqnarray*} 
\left|\mathcal{M}_{B(x,r)}(f)-\mathcal{M}_{B(y,r)}(f)\right|\lesssim
\left\{\begin{array}{ll}
\|f\|_{\ast,\kappa},\,\,\, &\|x-y\|\leq2r;\\
\,\\
\log_{2}(\|x-y\|/r)\|f\|_{\ast,\kappa},\,\,\, &\|x-y\|\geq2r;
\end{array}\right.
\end{eqnarray*}

$\rm{(b3)}$ there exist constants $c_{0},c_{1}>0$ independent of $f$ such that for all balls $B\subset\mathbb{R}^{N}$,
\begin{equation*}
  \omega\left(\{y\in{B}: |f(y)- \mathcal{M}_{B}(f)|>t\}\right)\leq c_{0}e^{-c_{1}t/\|f\|_{\ast,\kappa}}\omega(B),\qquad t>0.
\end{equation*}
\end{lemma}

\begin{proof} The results \rm{(b1)}-\rm{(b3)} were obtained in \cite{CS1977,GuoLiWick2026,JiuLi2023}. 
\end{proof}
 
To improve the readability and organization of the lengthy proof of Theorem \ref{thm-BMO-onesetp}, we present several key steps as a proposition and some lemmas.  
\begin{proof}[Proof of Theorem \ref{thm-BMO-onesetp}]
Since $f$ has compact support in $\mathbb{R}^{N}$, there exists $m_{0}\in{\mathbb{N}}$ such that $\supp f\subset B(0,2^{m_{0}-1})$.
Without loss of generality, we assume that $\int_{B(0,2^{m_{0}-1})}f(y)d\omega(y)=0$. In fact, let us define
\begin{equation*}
 \xi(x):=\frac{\omega(B(0,2^{m_{0}-1}))}{\omega(B(0,2^{m_{0}-2}))}\chi_{B(0,2^{m_{0}-2})}(x),\qquad x\in{\mathbb{R}^{N}},
\end{equation*}
where $\chi_{B(0,2^{m_{0}-2})}$ denotes the indicator of $B(0,2^{m_{0}-2})$. It is easy to obtain that
\begin{equation*}
\supp \xi\subset B(0,2^{m_{0}-1}), \quad \int_{\mathbb{R}^{N}}\xi(y)d\omega(y)=\omega(B(0,2^{m_{0}-1})) \quad \mathrm{and} \quad \|\xi\|_{L^{\infty}(\mathbb{R}^{N}, d\omega)}\leq2^{\mathbf{N}}.
\end{equation*}
Set $c:=\mathcal{M}_{B(0,2^{m_{0}-1})}(f)$. We have
\begin{equation*}
\int_{B(0,2^{m_{0}-1})}(f(y)-c\xi(y))d\omega(y)=0
\end{equation*}
and 
\begin{equation*} 
\supp (f-c\xi)\subset B(0,2^{m_{0}-1}),\qquad \|f-c\xi\|_{\ast,\kappa}+\|c\xi\|_{L^{\infty}(\mathbb{R}^{N}, d\omega)}\lesssim\|f\|_{\ast,\kappa}.
\end{equation*}
Then we replace $f$ by $f-c\xi$. The term $c\xi$ can be absorbed into the bounded function $g_{0}$ at the end of our proof.

Let $\mathcal{D}$ denote the system of dyadic cubes in $Q(0,2^{m_{0}})$, where
\begin{equation*}
  Q(0,2^{m_{0}})=\left\{y\in{\mathbb{R}^{N}}: \, -2^{m_{0}-1}\leq y_{j}<2^{m_{0}-1},\,\, j=1,2,\ldots,N\right\}.
\end{equation*}
Let $\mathcal{A}$ be a large positive integer which will be determined at the end of the proof. 
We next consider, successively in $n$, the special dyadic cubes $Q$ of side length
\begin{equation*}
  \ell_{n}:=2^{m_{0}-n\mathcal{A}},\qquad n\in{\mathbb{N}}.
\end{equation*}
Define $\ell_{0}:=2^{m_{0}}$. For every $n\in{\mathbb{N}}$, we denote 
\begin{equation*}
  \mathcal{D}_{n}:=\{Q\in{\mathcal{D}}:\,\, \ell(Q)=\ell_{n}\}.
\end{equation*}
Let $K_{t}(x,y)$ be as in \eqref{eq:approximating-kernel}, and let $C_{5}$ be as in $\rm{(a4)}$ of Lemma \ref{lem-kernel-properties}. We provide a key proposition as follows.
\begin{proposition}\label{pro:function-bn}
There exists a sequence of functions $\{b_{n}\}_{n\in{\mathbb{N}}}$ having the following basic properties:

$\rm(c1)$ For every $n\in{\mathbb{N}}$,  
\begin{equation*}
  |b_{n}(y)|\leq C_{6}\mathcal{A},\qquad y\in{\mathbb{R}^{N}},
\end{equation*}
where $C_{6}>0$ is a constant independent of $n$ and $\mathcal{A}$.

$\rm(c2)$ The functions $b_{n}$ have disjoint supports.

For every $n\in{\mathbb{N}}$, let us denote
\begin{equation*}
  U_{n}(x):=\int_{\mathbb{R}^{N}}K_{\ell_{n}}(x,y)b_{n}(y)d\omega(y),\qquad x\in{\mathbb{R}^{N}},
\end{equation*}
and
\begin{equation*}
  f_{n}:=f-\sum\limits_{i=1}^{n-1}U_{i},\qquad f_{1}:=f.
\end{equation*}
Let $n\in{\mathbb{N}}$. Then:

$\rm(c3)$ $|\mathcal{M}_{Q}(f_{n})|\leq\mathcal{A}$, for every dyadic cube $Q\subset\mathbb{R}^{N}$ with $\ell(Q)\geq\ell_{n-1}$.

$\rm(c4)$ For every dyadic cube $Q\subset\mathbb{R}^{N}$ with $\ell(Q)=\ell_{n}$, if
\begin{equation*}
|\mathcal{M}_{Q}(f_{n})|\leq\frac{9C_{5}+8}{20C_{5}}\mathcal{A},
\end{equation*}
then $U_{n}=0$ and $b_{n}=0$ on $Q$.

$\rm(c5)$ If $Q\in{\mathcal{D}_{n}}$ and $b_{n}(x)\neq0$ for some $x\in{Q}$, then 
\begin{equation*}
|\mathcal{M}_{Q}(f_{n+1})|\leq\frac{9C_{5}+7}{20C_{5}}\mathcal{A}.
\end{equation*}
\end{proposition}
\begin{proof}
Let $n\in{\mathbb{N}}$. Let us now suppose that $b_{1}, b_{2},\ldots,b _{n-1}$ have been constructed satisfying $\rm(c1)-\rm(c5)$. Before constructing the function $b_{n}$, we first establish some useful estimates in the following. Let $Q(x, \ell_{n}),Q(z, \ell_{n})$ be dyadic cubes, where $x,z\in{\mathbb{R}^{N}}$. We show that  
\begin{eqnarray}\label{eq:estimate-fn}
\left|\mathcal{M}_{Q(x, \ell_{n})}(f_{n})-\mathcal{M}_{Q(z, \ell_{n})}(f_{n})\right|
&\lesssim
\left\{\begin{array}{ll}
1,\,\, &\|x-z\|<2\ell_{n};\\
\,\\
1+\alpha\mathcal{A},\,\, &\left(2^{\alpha\mathcal{A}}+\sqrt{N}\right)\ell_{n}\geq\|x-z\|\geq2\ell_{n},
\end{array}\right.
\end{eqnarray}
if $0<\alpha\leq1/2$ satisfies $\mathcal{A}\geq1/\alpha$. Here the implicit constant denoted by $C_{7}$ depends only on $\mathbf{N}$ and is independent of $\mathcal{A}$. Choose $\alpha$ sufficiently small such that $\alpha C_{7}<1/40C_{5}$, and then take $\mathcal{A}\geq\max\{40C_{5}C_{7},1/\alpha\}$ sufficiently large so that $2^{\alpha\mathcal{A}}/2\geq\mathcal{A}^{4}$. Then 
\begin{equation}\label{eq:C7-C5-A}
C_{7}(1+\alpha\mathcal{A})<\frac{\mathcal{A}}{20C_{5}}.
\end{equation}
  
Indeed, we write
\begin{align*}
  \left|\mathcal{M}_{Q(x, \ell_{n})}(f_{n})-\mathcal{M}_{Q(z, \ell_{n})}(f_{n})\right|
  &\leq\left|\mathcal{M}_{Q(x, \ell_{n})}(f)-\mathcal{M}_{Q(z, \ell_{n})}(f)\right|+\left|\mathcal{M}_{Q(x, \ell_{n})}\left(\sum\limits_{i=1}^{n-1}U_{i}\right)-\mathcal{M}_{Q(z, \ell_{n})}\left(\sum\limits_{i=1}^{n-1}U_{i}\right)\right|\\
  &=: \mathcal{I}_{1}+\mathcal{I}_{2}.
\end{align*}
To estimate $\mathcal{I}_{1}$, by $\rm(b2)$ in Lemma \ref{lem-bmo-properties}, we have 
\begin{eqnarray*} 
\mathcal{I}_{1}\lesssim
\left\{\begin{array}{ll}
\|f\|_{\ast,\kappa},\,\,\, &\|x-z\|\leq2\ell_{n};\\
\,\\
\log_{2}(\|x-z\|/\ell_{n})\|f\|_{\ast,\kappa},\,\,\, &\|x-z\|\geq2\ell_{n}.
\end{array}\right.
\end{eqnarray*}

Consider $\mathcal{I}_{2}$. Let $i\in\{1,2,\ldots,n-1\}$. By $\rm(a3)$ in Lemma \ref{lem-kernel-properties}, $\rm(c1)$ and the definition of $U_{i}$, we have that for any $x,z\in{\mathbb{R}^{N}}$,
\begin{align}\label{eq:estimate-Ui}
  \left|U_{i}(x)-U_{i}(z)\right|
  &\lesssim\sum_{j=1}^{N}\int_{0}^{1}\int_{\mathbb{R}^{N}}\|x-z\|\left|\partial_{(\xi_{s})_{j}}K_{\ell_{i}}(\xi_{s},y)\right||b_{i}(y)|d\omega(y)ds\\\notag
  &\lesssim_{C_{1},C_{2},C_{4},C_{6}}\mathcal{A}\ell_{i}^{-1}\int_{0}^{1}\int_{\mathbb{R}^{N}}\frac{\|x-z\|}{\omega(B(\xi_{s}, \ell_{i}+d(\xi_{s},y)))}\frac{\ell_{i}}{\ell_{i}+\|\xi_{s}-y\|}d\omega(y)ds\\\notag
  &\lesssim_{C_{1},C_{2},C_{4},C_{6}} \mathcal{A}\ell_{i}^{-1}\|x-z\|,\notag
\end{align}
where $\xi_{s}=sx+(1-s)z$. Then
\begin{align*}
\left|\sum\limits_{i=1}^{n-1}(U_{i}(z)-U_{i}(x))\right|
&\lesssim_{C_{1},C_{2},C_{4},C_{6}}\sum\limits_{i=1}^{n-1}\mathcal{A}\ell_{i}^{-1}\|x-z\|\lesssim_{C_{1},C_{2},C_{4},C_{6}}\mathcal{A}\ell_{n-1}^{-1}\|x-z\|\\
&\lesssim_{C_{1},C_{2},C_{4},C_{6}}\mathcal{A}2^{-\mathcal{A}}\ell_{n}^{-1}\|x-z\|.
\end{align*}
It follows from this estimate that
\begin{align*}
  \mathcal{I}_{2}
  &\leq\left|\mathcal{M}_{Q(x, \ell_{n})}\left(\sum\limits_{i=1}^{n-1}U_{i}\right)-\sum\limits_{i=1}^{n-1}U_{i}(z)\right|+\left|\sum\limits_{i=1}^{n-1}U_{i}(z)-\sum\limits_{i=1}^{n-1}U_{i}(x)\right|+\left|\mathcal{M}_{Q(z, \ell_{n})}\left(\sum\limits_{i=1}^{n-1}U_{i}\right)-\sum\limits_{i=1}^{n-1}U_{i}(x)\right|\\
  &\lesssim_{C_{1},C_{2},C_{4},C_{6}}\mathcal{A}2^{-\mathcal{A}}\ell_{n}^{-1}\|x-z\|.
\end{align*}
Combining the estimates of $\mathcal{I}_{1}$ and $\mathcal{I}_{2}$, we see that \eqref{eq:estimate-fn} holds. It means that on the scale $\ell_{n}$, the mean value changes slowly even for the distance $\|x-z\|=\left(2^{\alpha\mathcal{A}}+\sqrt{N}\right)\ell_{n}$.

Now to construct the function $b_{n}$, we split the argument into the following four steps. 

{\bf{Step I:}} An auxiliary function $\beta_{n}$.  Consider the following set 
\begin{equation*}
  \Omega_{n}:=\left\{Q\in{\mathcal{D}_{n}}:\,\, |\mathcal{M}_{Q}(f_{n})|\geq\frac{9C_{5}+10}{20C_{5}}\mathcal{A}\right\}.
\end{equation*}
If $\Omega_{n}=\emptyset$, we define $b_{n}(y):=0$, $y\in{\mathbb{R}^{N}}$. It is obvious that $\rm(c1)-\rm(c5)$ hold for the index $n$.

Suppose first that $\Omega_{n}\neq\emptyset$. Then at least one of the following two sets is nonempty:
\begin{equation*}
  \Omega^{+}_{n}:=\left\{Q\in{\Omega_{n}}:\,\mathcal{M}_{Q}(f_{n})\geq\frac{9C_{5}+10}{20C_{5}}\mathcal{A}\right\}\quad \mathrm{and} \quad \Omega^{-}_{n}:=\left\{Q\in{\Omega_{n}}:\,\mathcal{M}_{Q}(f_{n})\leq-\frac{9C_{5}+10}{20C_{5}}\mathcal{A}\right\}.
\end{equation*}
We first consider the set $\Omega^{+}_{n}$. Take a cube $Q\in{\Omega^{+}_{n}}$. Choose $Q'\in{\mathcal{D}}$ with $\ell(Q')=\ell_{n-1}$ such that $Q\subset Q'$. By \eqref{eq:estimate-Ui}, $\rm(c3)$ and Lemma \ref{lem-bmo-properties} $\rm(b1)$, we have 
\begin{equation*} 
\mathcal{M}_{Q}(f_{n})\leq|\mathcal{M}_{Q}(f_{n})-\mathcal{M}_{Q'}(f_{n})|+|\mathcal{M}_{Q'}(f_{n})|\lesssim\mathcal{A}.
\end{equation*}
This together with the definition of $\Omega^{+}_{n}$ implies that $\mathcal{M}_{Q}(f_{n})\thicksim \mathcal{A}$, for every $Q\in{\Omega^{+}_{n}}$. 

For each $Q\in{\Omega^{+}_{n}}$, let $\Lambda^{+}(Q)$ be the union of the cubes in the maximal connected subfamily of $\mathcal{D}_{n}$ containing $Q$ such that
\begin{equation*}
\mathcal{M}_{w}(f_{n})\geq\frac{9C_{5}+9}{20C_{5}}\mathcal{A},\qquad \text{for every}\,\,\, w\subseteq\Lambda^{+}(Q)\quad\mathrm{with}\quad w\in{\mathcal{D}_{n}}.
\end{equation*}
Here, a subfamily of $\mathcal{D}_{n}$ is called connected if any two of its cubes can be joined by a finite chain of adjacent cubes belonging to that subfamily. Two cubes $w,w'\in{\mathcal{D}_{n}}$ are called adjacent if $\overline{w}\cap\overline{w'}\neq\emptyset$.

Define $\Sigma^{+}_{n}:=\cup_{Q\in{\Omega^{+}_{n}}}\Lambda^{+}(Q)$. Using this set, we define
\begin{eqnarray*} 
\beta_{n}^{+}(y):=
\left\{\begin{array}{ll}
\mathcal{M}_{Q}(f_{n}),\quad &y\in Q\subset{\Sigma^{+}_{n}},\,\,Q\in{\mathcal{D}_{n}};\\
\,\\
0,\quad &\mathrm{otherwise},
\end{array}\right.
\end{eqnarray*}
and 
\begin{equation*}
  V_{n}^{+}(y):=\int_{\mathbb{R}^{N}}K_{\ell_{n}}(y,z)\beta_{n}^{+}(z)d\omega(z),\qquad y\in{\mathbb{R}^{N}}.
\end{equation*}
Notice that by the definition of $\Sigma^{+}_{n}$, we have that $\{y\in{\mathbb{R}^{N}}:\,\beta^{+}_{n}(y)\neq0\}=\Sigma^{+}_{n}$.

We now give a lower bound for the function $V_{n}^{+}$ on $\Sigma^{+}_{n}$. 
\begin{lemma}\label{lem:lowbounded-kernel}
For every $x\in{\Sigma^{+}_{n}}$,
\begin{equation}\label{eq:estimate-Vn}
V^{+}_{n}(x)\geq\frac{\left(4+\frac{4}{C_{5}}\right)\mathcal{A}}{20C_{5}}-C_{7}(1+\alpha\mathcal{A})\geq\frac{4\mathcal{A}}{20C_{5}}-C_{7}(1+\alpha\mathcal{A})>0.
\end{equation}
\end{lemma}
\begin{proof} 
If $x\in{\Sigma^{+}_{n}}$, then there exists a dyadic cube $Q':=Q(x_{0}, \ell_{n})\subset\Sigma^{+}_{n}$ such that $x\in{Q'}$. Notice that by the definitions of $K_{\ell_{n}}(x,y)$ and $\beta_{n}^{+}$, we have
\begin{align*} 
  \left|V_{n}^{+}(x)-\widetilde{K}_{n}(x)\mathcal{M}_{Q'}(f_{n})\right| 
  &\leq\int_{B(x,2\mathcal{A}^{4}\ell_{n})\cap\Sigma^{+}_{n}}K_{\ell_{n}}(x,y)|\beta_{n}^{+}(y)-\mathcal{M}_{Q'}(f_{n})|d\omega(y),
\end{align*}
where 
\begin{equation*}
\widetilde{K}_{n}(x):=\int_{B(x,2\mathcal{A}^{4}\ell_{n})\cap\Sigma^{+}_{n}}K_{\ell_{n}}(x,y)d\omega(y).
\end{equation*}
Applying \eqref{eq:estimate-fn},  
\begin{equation*}
\left|V_{n}^{+}(x)-\widetilde{K}_{n}(x)\mathcal{M}_{Q'}(f_{n})\right| \leq C_{7}(1+\alpha\mathcal{A})\int_{B(x,2\mathcal{A}^{4}\ell_{n})}K_{\ell_{n}}(x,y)d\omega(y)\leq C_{7}(1+\alpha\mathcal{A}).
\end{equation*}
It implies that 
\begin{equation*}  
V_{n}^{+}(x)\geq\frac{(9C_{5}+9)\mathcal{A}\widetilde{K}_{n}(x)}{20C_{5}}-C_{7}(1+\alpha\mathcal{A}).
\end{equation*}

We now show an essential estimate of $\widetilde{K}_{n}$. Note that
\begin{equation*}
Q'\subset B(x,2\mathcal{A}^{4}\ell_{n})\cap\Sigma^{+}_{n}.
\end{equation*}
Then by $\rm{(a4)}$ of Lemma \ref{lem-kernel-properties}, for every $x\in{Q'}$,
\begin{equation*}
\widetilde{K}_{n}(x)\geq\int_{B(x,2\mathcal{A}^{4}\ell_{n})\cap\Sigma^{+}_{n}}K_{\ell_{n}}(x,y)d\omega(y)\geq\int_{Q'}K_{\ell_{n}}(x,y)d\omega(y)\geq\frac{4}{9C_{5}}.
\end{equation*}
Hence, from these facts and \eqref{eq:C7-C5-A}, Lemma \ref{lem:lowbounded-kernel} follows.
\end{proof}

Using $\Omega^{-}_{n}$, we define the sets $\Lambda^{-}(I)$ and $\Sigma^{-}_{n}$, the function $\beta_{n}^{-}$ on $\Sigma^{-}_{n}$, and the function $V_{n}^{-}$ in exactly the same way as $\Lambda^{+}(I)$, $\Sigma^{+}_{n}$, $\beta_{n}^{+}$ and $V_{n}^{+}$ are defined using $\Omega^{+}_{n}$. We also have $\{y\in{\mathbb{R}^{N}}:\,\beta^{-}_{n}(y)\neq0\}=\Sigma^{-}_{n}$. Following the same argument as in the proof of Lemma \ref{lem:lowbounded-kernel}, we also obtain
\begin{equation*} 
V_{n}^{-}(x)\leq-\frac{4\mathcal{A}}{20C_{5}}+C_{7}(1+\alpha\mathcal{A}),\qquad x\in{\Sigma^{-}_{n}}.
\end{equation*}
Let us define 
\begin{equation*}
\beta_{n}(y):=\beta_{n}^{+}(y)+\beta_{n}^{-}(y)\quad \mathrm{and}\quad V_{n}(y):=\int_{\mathbb{R}^{N}}K_{\ell_{n}}(y,z)\beta_{n}(z)d\omega(z),\quad y\in{\mathbb{R}^{N}}.
\end{equation*}
Then $V_{n}(y)=V_{n}^{+}(y)+V_{n}^{-}(y)$, $y\in{\mathbb{R}^{N}}$.

{\bf{Step II:}} The functions $\beta_{n}$, $f_{n}$ and $f_{n}-V_{n}$ satisfy $\rm(c3)-\rm(c5)$. Let $\Sigma_{n}:=\Sigma^{+}_{n}\cup\Sigma^{-}_{n}$. We now establish the following statements. For any sets $A,B\subset\mathbb{R}^{N}$, we define $\dist(A,B):=\inf_{x\in{A},\,y\in{B}}\|x-y\|$.
\begin{lemma}\label{lem:cube-estimate}
The following statements hold.
\begin{enumerate}
\item [$\rm(d1)$]
\begin{equation*}
  \dist(\Sigma^{+}_{n},\Sigma^{-}_{n})>2\ell_{n}\mathcal{A}^{4},
\end{equation*}
and hence, for every $y\in{\Sigma_{n}}$, $V_{n}^{+}(y)\neq0$ and $V_{n}^{-}(y)\neq0$ are not simultaneously true.

\item [$\rm(d2)$]
For any $Q\subset\Sigma_{n}$ with $Q\in{\mathcal{D}_{n}}$,  
\begin{equation}\label{eq:main-estimate1}
|\mathcal{M}_{Q}(f_{n})-V_{n}(x)|\leq\frac{(9C_{5}+6)\mathcal{A}}{20C_{5}}+C_{7}(1+\alpha\mathcal{A}),\qquad x\in{Q}.
\end{equation}

\item [$\rm(d3)$]
 For every dyadic cube $Q\subset\mathbb{R}^{N}$ with $\ell(Q)=\ell_{n}$, if $|\mathcal{M}_{Q}(f_{n})|\leq(9C_{5}+8)\mathcal{A}/20C_{5}$, then  
\begin{equation}\label{eq:main-estimate2}
 V_{n}(x)=\beta_{n}(x)=0,\qquad x\in{Q}.
\end{equation}

\item [$\rm(d4)$] For each dyadic cube $Q\subset\mathbb{R}^{N}$ with $\ell(Q)\geq\ell_{n}$,
\begin{equation*} 
|\mathcal{M}_{Q}(f_{n}-V_{n})|\leq\frac{(9C_{5}+10)\mathcal{A}}{20C_{5}}.
\end{equation*}
\end{enumerate}
\end{lemma}
\begin{proof} 
$\rm(d1)$. Suppose $\dist(\Sigma^{+}_{n},\Sigma^{-}_{n})\leq2\mathcal{A}^{4}\ell_{n}$. Then there exist $Q\subset \Sigma^{+}_{n}$ and $Q'\subset \Sigma^{-}_{n}$, with $\ell(Q)=\ell(Q')=\ell_n$ such that $\dist(Q,Q')=\dist(\Sigma^{+}_{n},\Sigma^{-}_{n})\leq2\mathcal{A}^{4}\ell_{n}$. It then follows from \eqref{eq:estimate-fn} that  
\begin{equation*}
  |\mathcal{M}_{Q}(f_{n})-\mathcal{M}_{Q'}(f_{n})|\leq C_{7}(1+\alpha\mathcal{A}).
\end{equation*}
From the definitions of $\Sigma^{+}_{n}$ and $\Sigma^{-}_{n}$,
\begin{equation*}
|\mathcal{M}_{Q}(f_{n})-\mathcal{M}_{Q'}(f_{n})|\geq\frac{9C_{5}+9}{10C_{5}}\mathcal{A}.
\end{equation*}
By \eqref{eq:C7-C5-A}, we conclude that $\rm(d1)$ holds.

$\rm(d2)$. Firstly, we claim that, for any cube $Q\subset\Sigma_n$ with
$Q\in{\mathcal{D}_{n}}$, if
$|\mathcal{M}_Q(f_n)|\geq(9C_{5}+10)\mathcal{A}/20C_{5}$, then
\begin{equation}\label{eq:bad-cube0}
|V_n(x)-\mathcal{M}_Q(f_n)|
\leq C_7(1+\alpha\mathcal{A}),
\qquad x\in Q.
\end{equation}

First, fix $Q:=Q(x_{0},\ell_{n})\in{\mathcal{D}_{n}}$ such that $Q\subset\Sigma_{n}^{+}$ and $\mathcal{M}_{Q}(f_{n})\geq(9C_{5}+10)\mathcal{A}/20C_{5}$. If a dyadic cube $Q':=Q(x'_{0},\ell_{n})\subset\mathbb{R}^{N}$ has center $x'_{0}$ satisfying $\|x_{0}-x'_{0}\|\leq\left(\sqrt{N}+2\mathcal{A}^{4}\right)\ell_{n}$, then $Q'\subset\Sigma_{n}^{+}$. Indeed, it follows from \eqref{eq:estimate-fn} that 
\begin{equation*}
\left|\mathcal{M}_{Q}(f_{n})-\mathcal{M}_{Q'}(f_{n})\right|\leq C_{7}(1+\alpha\mathcal{A}).
\end{equation*}
Then by \eqref{eq:C7-C5-A},  
\begin{equation*}
\mathcal{M}_{Q'}(f_{n})\geq\frac{(9C_{5}+9)\mathcal{A}}{20C_{5}}.
\end{equation*}
On the other hand, since $b_{1}, b_{2},\ldots,b _{n-1}$ satisfy $\rm(c1)-\rm(c5)$, it follows that
\begin{equation}\label{eq:fn=0-R-Q}
  f_{n}(x)=0,\qquad \text{for\,\, a.e. }x\in{\mathbb{R}^{N}\backslash Q(0,2^{m_{0}})}.
\end{equation}
In fact, note that $f_{1}=f=0$ on $\mathbb{R}^{N}\backslash Q(0,2^{m_{0}})$. For any dyadic cube $Q_{1}\subset\mathbb{R}^{N}\backslash Q(0,2^{m_{0}})$ with $\ell(Q_{1})=\ell_{1}$, we have $\mathcal{M}_{Q_{1}}(f_{1})=0$. Then by this fact and $\rm(c4)$, 
\begin{equation*}
  U_{1}(x)=b_{1}(x)=0,\qquad \text{for\,\, a.e. }x\in{\mathbb{R}^{N}\backslash Q(0,2^{m_{0}})}.
\end{equation*}
This implies that $f_{2}(x)=0$, for a.e. $x\in{\mathbb{R}^{N}\backslash Q(0,2^{m_{0}})}$. Using $\rm(c4)$ again, we have
\begin{equation*}
  U_{2}(x)=b_{2}(x)=0,\qquad \text{for\,\, a.e. }x\in{\mathbb{R}^{N}\backslash Q(0,2^{m_{0}})}.
\end{equation*}
Iterating this argument up to $n-1$, we obtain \eqref{eq:fn=0-R-Q}.
Therefore, by \eqref{eq:fn=0-R-Q} and the definitions of $\Omega^{+}_{n}$ and $\Lambda^{+}(Q)$, we have $Q'\subset\Lambda^{+}(Q)$, and thus $Q'\subset\Sigma_{n}^{+}$.

Fix $x\in{Q(x_{0},\ell_{n})}$. For each $y\in{B(x,2\mathcal{A}^{4}\ell_{n})}$, there exists a dyadic cube $Q'\subset\mathbb{R}^{N}$ with $\ell(Q')=\ell_{n}$ such that $y\in{Q'}$. By the preceding observation, $Q'\subset\Sigma_{n}^{+}$. Then
\begin{equation*}
  |V_{n}(x)-\mathcal{M}_{Q}(f_{n})|\leq\int_{\mathbb{R}^{N}}K_{\ell_{n}}(x,y)|\beta_{n}(y)-\mathcal{M}_{Q}(f_{n})|d\omega(y)\leq\sup\limits_{\|x-y\|\leq2\ell_{n}\mathcal{A}^{4}}\left|\beta_{n}(y)-\mathcal{M}_{Q}(f_{n})\right|\leq C_{7}(1+\alpha\mathcal{A}).
\end{equation*}
This proves \eqref{eq:bad-cube0}. On the other hand, if $Q\subset\Sigma_{n}^{-}$ and $\mathcal{M}_{Q}(f_{n})\leq-(9C_{5}+10)\mathcal{A}/20C_{5}$, an analogous argument yields \eqref{eq:bad-cube0}.

We now prove \eqref{eq:main-estimate1}. Let $Q\subset\Sigma_{n}^{+}$ with $Q\in{\mathcal{D}_{n}}$. Then $\mathcal{M}_{Q}(f_{n})\geq(9C_{5}+9)\mathcal{A}/20C_{5}$. We consider the following two cases.

Case 1: $\mathcal{M}_{Q}(f_{n})\geq(9C_{5}+10)\mathcal{A}/20C_{5}$. By \eqref{eq:bad-cube0}, we deduce that 
\begin{equation*}
  |\mathcal{M}_{Q}(f_{n})-V_{n}(x)|\leq C_{7}(1+\alpha\mathcal{A}),\qquad x\in{Q}.
\end{equation*}

Case 2: $(9C_{5}+9)\mathcal{A}/20C_{5}\leq\mathcal{M}_{Q}(f_{n})<(9C_{5}+10)\mathcal{A}/20C_{5}$. It follows from \eqref{eq:estimate-Vn} that 
\begin{equation*}
V_{n}(x)\geq \frac{4\mathcal{A}}{20C_{5}}-C_{7}(1+\alpha\mathcal{A}),\qquad x\in{Q}.
\end{equation*}
Then for any $x\in{Q}$,
\begin{equation*}
\mathcal{M}_{Q}(f_{n})-V_{n}(x)\leq \frac{(9C_{5}+10)\mathcal{A}}{20C_{5}}-\left(\frac{4\mathcal{A}}{20C_{5}}-C_{7}(1+\alpha\mathcal{A})\right)\leq\frac{(9C_{5}+6)\mathcal{A}}{20C_{5}}+C_{7}(1+\alpha\mathcal{A}).
\end{equation*}
Furthermore, by \eqref{eq:estimate-fn}, we have
\begin{equation*}
 V_{n}(x)\leq\mathcal{M}_{Q}(f_{n})+C_{7}(1+\alpha\mathcal{A}),\qquad x\in{Q}.
\end{equation*}
It then follows that for any $x\in{Q}$,
\begin{equation*}
\mathcal{M}_{Q}(f_{n})-V_{n}(x)\geq \mathcal{M}_{Q}(f_{n})-\mathcal{M}_{Q}(f_{n})-C_{7}(1+\alpha\mathcal{A})\geq-C_{7}(1+\alpha\mathcal{A}).
\end{equation*}
Thus we conclude that \eqref{eq:main-estimate1} holds.
Moreover, if $Q\subset\Sigma_{n}^{-}$ with $Q\in{\mathcal{D}_{n}}$, by an argument analogous to Cases 1 and 2, we also have \eqref{eq:main-estimate1}. 

$\rm(d3)$. Let $Q\subset\mathbb{R}^{N}$ be a dyadic cube satisfying $\ell(Q)=\ell_{n}$ and $|\mathcal{M}_{Q}(f_{n})|\leq(9C_{5}+8)\mathcal{A}/20C_{5}$. By the definitions of $\beta_{n}$ and $\Omega_{n}$, we have $\beta_{n}=0$ on $Q$. Assume $V_{n}(x)\neq0$, for some $x\in{Q}$. Then 
\begin{equation}\label{eq:vn-d0}
0<|V_{n}(x)|\leq\int_{B(x,2\mathcal{A}^{4}\ell_{n})}K_{\ell_{n}}(x,y)|\beta_{n}(y)|d\omega(y).
\end{equation}
It implies that $\mathcal{S}_{n}(x)\neq\emptyset$, where
\begin{equation}\label{eq:supp-betan}
\mathcal{S}_{n}(x):=\Sigma_{n}\cap B(x,2\mathcal{A}^{4}\ell_{n}).
\end{equation}
Note that by the definitions of $\beta_{n}$ and $\Sigma_{n}$,
\begin{equation*}
  \{y\in{\mathbb{R}^{N}}:\,\beta_{n}(y)\neq0\}=\Sigma_{n}.
\end{equation*}
Then for every $y\in{\mathcal{S}_{n}(x)}$, there exists a cube $Q'\subset\Sigma_{n}$ with $Q'\in{\mathcal{D}_{n}}$ such that $y\in{Q'}$. By \eqref{eq:estimate-fn},  
\begin{equation}\label{eq:bad-cube}
  |\beta_{n}(y)-\mathcal{M}_{Q}(f_{n})|=|\mathcal{M}_{Q'}(f_{n})-\mathcal{M}_{Q}(f_{n})|\leq C_{7}(1+\alpha\mathcal{A}).
\end{equation}
On the other hand, since $|\mathcal{M}_{Q}(f_{n})|\leq(9C_{5}+8)\mathcal{A}/20C_{5}$ and $|\mathcal{M}_{Q'}(f_{n})|\geq(9C_{5}+9)\mathcal{A}/20C_{5}$, we have
\begin{equation*}
  \left|\mathcal{M}_{Q'}(f_{n})-\mathcal{M}_{Q}(f_{n})\right|\geq \frac{(9C_{5}+9)\mathcal{A}}{20C_{5}}-\frac{(9C_{5}+8)\mathcal{A}}{20C_{5}}=\frac{\mathcal{A}}{20C_{5}}.
\end{equation*}
These facts yield a contradiction by \eqref{eq:C7-C5-A}. Moreover, by \eqref{eq:bad-cube}, we also have $|\mathcal{M}_{Q'}(f_{n})|<(9C_{5}+9)\mathcal{A}/20C_{5}$. These facts imply that $\beta_{n}(y)=0$, for $y\in B(x,2\mathcal{A}^{4}\ell_{n})$. Hence, we obtain a contradiction by \eqref{eq:vn-d0}. We then conclude that $V_{n}=0$ on $Q$. Therefore, \eqref{eq:main-estimate2} holds. 

$\rm(d4)$. Let $Q\subset\mathbb{R}^{N}$ be a dyadic cube and $\ell(Q)=\ell_{n}$. Consider the following three cases.

Case 1: $|\mathcal{M}_{Q}(f_{n})|\leq(9C_{5}+8)\mathcal{A}/20C_{5}$. By $\rm(d3)$,
\begin{equation*}
|\mathcal{M}_{Q}(f_{n}-V_{n})|=|\mathcal{M}_{Q}(f_{n})|\leq(9C_{5}+8)\mathcal{A}/20C_{5}.
\end{equation*}

Case 2: $|\mathcal{M}_{Q}(f_{n})|>(9C_{5}+8)\mathcal{A}/20C_{5}$ and $Q\subset\Sigma_{n}$. From $\rm(d2)$ and \eqref{eq:C7-C5-A}, we get 
\begin{equation*}
|\mathcal{M}_{Q}(f_{n}-V_{n})|\leq(9C_{5}+6)\mathcal{A}/20C_{5}+C_{7}(1+\alpha\mathcal{A})\leq(9C_{5}+7)\mathcal{A}/20C_{5}.
\end{equation*}

Case 3: $|\mathcal{M}_{Q}(f_{n})|>(9C_{5}+8)\mathcal{A}/20C_{5}$ and $Q\not\subset\Sigma_{n}$. First, we claim that 
\begin{equation*}
|\mathcal{M}_{Q}(f_{n})|\leq(9C_{5}+10)\mathcal{A}/20C_{5}.
\end{equation*}
Indeed, suppose $|\mathcal{M}_{Q}(f_{n})|>(9C_{5}+10)\mathcal{A}/20C_{5}$. By the definitions of $\Omega_{n}$ and $\Sigma_{n}$, we have $Q\subset\mathbb{R}^{N}\backslash Q(0,2^{m_{0}})$. It then follows from \eqref{eq:fn=0-R-Q} that $\mathcal{M}_Q(f_n)=0$, which contradicts our assumption. Hence, the claim holds.
 
We next prove
\begin{equation}\label{eq:Caes3-f-v}
|\mathcal{M}_{Q}(f_{n}-V_{n})|\leq\frac{(9C_{5}+10)\mathcal{A}}{20C_{5}}.
\end{equation}
Suppose first that $\dist(Q,\Sigma_{n})>2\mathcal{A}^{4}\ell_{n}$. Then we have $V_{n}=0$ on $Q$. This fact together with the above claim yields \eqref{eq:Caes3-f-v}.

We now consider the remaining case $\dist(Q,\Sigma_{n})\leq2\mathcal{A}^{4}\ell_{n}$. Assume first that $\mathcal{M}_{Q}(f_{n})>(9C_{5}+8)\mathcal{A}/20C_{5}$. 
Let $x\in{Q}$ be such that $\mathcal{S}_{n}(x)\neq\emptyset$, where $\mathcal{S}_{n}(x)$ is as in \eqref{eq:supp-betan}.
By \eqref{eq:estimate-fn}, we have that for every $y\in{\mathcal{S}_{n}(x)}$,
\begin{equation*} 
  |\beta_{n}(y)-\mathcal{M}_{Q}(f_{n})|\leq C_{7}(1+\alpha\mathcal{A}).
\end{equation*}
Hence, we see
\begin{equation*} 
0\leq\mathcal{M}_{Q}(f_{n})-C_{7}(1+\alpha\mathcal{A})\leq\beta_{n}(y)\leq \mathcal{M}_{Q}(f_{n})+C_{7}(1+\alpha\mathcal{A}),\qquad y\in{\mathcal{S}_{n}(x)}.
\end{equation*}
It follows that
\begin{equation*}
0\leq V_{n}(x)\leq \mathcal{M}_{Q}(f_{n})+C_{7}(1+\alpha\mathcal{A}),\qquad x\in{Q}.
\end{equation*}
Moreover, if $x\in{Q}$ and $\mathcal{S}_{n}(x)=\emptyset$, then $V_{n}(x)=0$.
These facts yield
\begin{equation*}
0\leq \mathcal{M}_{Q}(V_{n})\leq \mathcal{M}_{Q}(f_{n})+C_{7}(1+\alpha\mathcal{A}).
\end{equation*}
By this estimate, we deduce that \eqref{eq:Caes3-f-v} holds. On the other hand, if $\mathcal{M}_{Q}(f_{n})<-(9C_{5}+8)\mathcal{A}/20C_{5}$, by an argument similar to the one above, we also obtain \eqref{eq:Caes3-f-v}.

For each dyadic cube $Q\subset\mathbb{R}^{N}$ satisfying $\ell(Q)\geq\ell_{n}$, there exists a finite collection of pairwise disjoint dyadic subcubes $w_{i}\subset\mathbb{R}^{N}$ with $i=1,2,\ldots,m(Q)$ such that
\begin{equation*}
 \ell(w_{i})=\ell_{n} \quad \text{and}\quad Q=\bigcup\limits_{i=1}^{m(Q)}w_{i},\quad \text{where}\quad m(Q)=\left(\frac{\ell(Q)}{\ell_{n}}\right)^{N}.
\end{equation*}
Then by Cases 1-3, 
\begin{equation*}
  |\mathcal{M}_{Q}(f_{n}-V_{n})|\leq\sum_{i=1}^{m(Q)}\frac{\omega(w_{i})}{\omega(Q)}|\mathcal{M}_{w_{i}}(f_{n}-V_{n})|
  \leq\frac{(9C_{5}+10)\mathcal{A}}{20C_{5}}\sum_{i=1}^{m(Q)}\frac{\omega(w_{i})}{\omega(Q)}\leq\frac{(9C_{5}+10)\mathcal{A}}{20C_{5}}.
\end{equation*}
Therefore, we complete the proof of Lemma \ref{lem:cube-estimate}. 
\end{proof}

From Lemma \ref{lem:cube-estimate}, if we substitute $\beta_{n}$ for $b_{n}$ in the definition of $f_{n+1}$, we see that $\rm(c1)$ and $\rm(c3)-\rm(c5)$ are satisfied for the index $n$. Next, we modify the function $\beta_{n}$ so that the property $\rm(c2)$ holds.

{\bf{Step III:}} Construction of the function $b_{n}$. Let $\mathcal{W}_{n}$ be the set of all cubes $Q\subset{\Sigma_{n}}$ with $Q\in{\mathcal{D}_{n}}$. Let $Q\in\mathcal{W}_{n}$ and denote
\begin{equation*}
 \Theta(Q):=\left\{\tilde{I}\subset{Q}:\,\,|\mathcal{M}_{\tilde{I}}(f)-\mathcal{M}_{Q}(f)|\geq \mathcal{A}/30C_{5},\,\tilde{I}\in{\mathcal{D}},\,\ell(\tilde{I})=\ell_{n+j},\,\,j\in{\mathbb{N}}\right\}.
\end{equation*}
Let $\{\tilde{J}_i\}_{i\in\Pi}$ denote the collection of cubes in
$\Theta(Q)$ that are maximal with respect to inclusion, where
$\Pi\subseteq\mathbb{N}$ is an at most countable index set. Thus, for
each $i\in\Pi$, there is no cube $\tilde{I}\in\Theta(Q)$ such that
\[
\tilde{J}_i\subsetneq\tilde{I}.
\]
By the nesting property of the dyadic grid, these maximal cubes are
pairwise disjoint:
\[
\tilde{J}_m\cap\tilde{J}_n=\varnothing,
\qquad \text{for}\,\, m,n\in\Pi,\quad m\neq n.
\]

Let $Q\in\mathcal{W}_{n}$ and $\mathcal{G}(Q):=\cup_{i\in{\Pi}}\tilde{J}_i$. We now define
\begin{eqnarray*} 
b_{n}(y):=
\left\{\begin{array}{ll}
\beta_{n}(y),\quad &y\in Q\backslash\mathcal{G}(Q),\,\,Q\in\mathcal{W}_{n};\\
\,\\
0,\quad &\mathrm{otherwise},
\end{array}\right.
\end{eqnarray*}
and 
\begin{equation*}
  U_{n}(x):=\int_{\mathbb{R}^{N}}K_{\ell_{n}}(x,y)b_{n}(y)d\omega(y),\quad x\in{\mathbb{R}^{N}}.
\end{equation*}

We establish an estimate for $|U_{n}-V_{n}|$ as follows.
\begin{lemma}\label{lem:U-V-estimate}
For any $x\in\mathbb{R}^{N}$,
\begin{equation*}
  |U_{n}(x)-V_{n}(x)|\lesssim_{C_{1},C_{4},C_{6}}c_{0}\mathcal{A}^{4N+1}e^{-\frac{c_{1}\mathcal{A}}{60C_{5}}},
\end{equation*}
where $c_{0},c_{1}$ are the absolute constants from $\rm(b3)$ of Lemma \ref{lem-bmo-properties}.
\end{lemma}
\begin{proof}
Let $Q\in\mathcal{W}_{n}$. We first claim that
\begin{equation}\label{eq:main-estimate4}
  \omega\left(\mathcal{G}(Q)\right)\leq 2c_{0}e^{-\frac{c_{1}\mathcal{A}}{60C_{5}}}\omega(Q),
\end{equation}
where $c_{0},c_{1}>0$ are constants in $\rm(b3)$ of Lemma \ref{lem-bmo-properties}, independent of $Q$.
Indeed, let $\{\tilde{J}_i\}_{i\in\Pi}\subset\Theta(Q)$ be as above. For any $i\in\Pi$, let us denote
\begin{equation*}
  \tilde{J}_{i}^{+}:=\{y\in{\tilde{J}_{i}}: |f(y)- \mathcal{M}_{\tilde{J}_{i}}(f)|>\mathcal{A}/60C_{5}\}
\end{equation*}
and 
\begin{equation*}
 \tilde{J}_{i}^{-}:=\{y\in{\tilde{J}_{i}}: |f(y)- \mathcal{M}_{\tilde{J}_{i}}(f)|\leq\mathcal{A}/60C_{5}\},
\end{equation*}
and 
\begin{equation*}
\mathcal{G}^{+}(Q):=\bigcup_{i\in{\Pi}}\tilde{J}_{i}^{+}\quad \mathrm{and}\quad\mathcal{G}^{-}(Q):=\bigcup_{i\in{\Pi}}\tilde{J}_{i}^{-}.
\end{equation*}
Firstly, by $\rm(b3)$ in Lemma \ref{lem-bmo-properties}, we have that
\begin{equation*}
  \omega\left(\mathcal{G}^{+}(Q)\right)\leq \sum\limits_{i\in{\Pi}}\omega\left(\tilde{J}_{i}^{+}\right)\leq c_{0}e^{-\frac{c_{1}\mathcal{A}}{60C_{5}}}\sum\limits_{i\in{\Pi}}\omega\left(\tilde{J}_{i}\right)\leq c_{0}e^{-\frac{c_{1}\mathcal{A}}{60C_{5}}}\omega(Q).
\end{equation*}
On the other hand, for any $x\in{\tilde{J}_{i}^{-}}$ with $i\in{\Pi}$, we deduce
\begin{equation*}
|f(x)-\mathcal{M}_{Q}(f)|\geq \mathcal{A}/60C_{5}.
\end{equation*}
It then follows that 
\begin{equation*}
\mathcal{G}^{-}(Q)\subset\left\{x\in{Q}:\,\,|f(x)-\mathcal{M}_{Q}(f)|\geq\mathcal{A}/60C_{5}\right\}.
\end{equation*}
Then using $\rm(b3)$ in Lemma \ref{lem-bmo-properties} again, we conclude that \eqref{eq:main-estimate4} holds.

We now turn to our main objective. It is easily verified that for $x\in{\mathbb{R}^{N}}$,
\begin{align*}
  |U_{n}(x)-V_{n}(x)|
  &=\left|\int_{\mathbb{R}^{N}}K_{\ell_{n}}(x,y)(b_{n}(y)-\beta_{n}(y))d\omega(y)\right|\\
  &=\left|\int_{\bigcup\limits_{Q\in\mathcal{W}_{n}}\mathcal{G}(Q)}K_{\ell_{n}}(x,y)\beta_{n}(y)d\omega(y)\right|\\
  &\lesssim_{C_{6}}\mathcal{A} \sum\limits_{Q\in\mathcal{W}_{n}}\int_{\mathcal{G}(Q)\bigcap\{y\in{\mathbb{R}^{N}}:\,\|x-y\|\leq2\ell_{n}\mathcal{A}^{4}\}}K_{\ell_{n}}(x,y)d\omega(y).
\end{align*}  
For each $x\in\mathbb{R}^{N}$, let $m(x)$ denote the cardinality of the
collection
\[
\left\{
Q\in\mathcal{W}_{n}:
\mathcal{G}(Q)\cap
\left\{y\in\mathbb{R}^{N}:\|x-y\|
\leq 2\ell_{n}\mathcal{A}^{4}\right\}
\neq\emptyset
\right\}.
\]
Enumerate all cubes in this collection as
\[
Q_{1}(x),\ldots,Q_{m(x)}(x).
\]
Then, for every $x\in\mathbb{R}^{N}$, $m(x)\leq (4\mathcal{A}^{4}+3)^{N}$. Consequently, for any $x\in\mathbb{R}^{N}$, by \eqref{eq:main-estimate4} and the fact that distinct cubes in $\mathcal{W}_n$ are disjoint,
\begin{align*}
  |U_{n}(x)-V_{n}(x)|
  &\lesssim_{C_{6}}\mathcal{A} \sum\limits_{i=1}^{m(x)}\int_{\mathcal{G}(Q_{i}(x))\bigcap\{y\in{\mathbb{R}^{N}}:\,\|x-y\|\leq2\ell_{n}\mathcal{A}^{4}\}}K_{\ell_{n}}(x,y)d\omega(y)\\
  &\lesssim_{C_{1},C_{4},C_{6}}\mathcal{A} \sum\limits_{i=1}^{m(x)}\int_{\mathcal{G}(Q_{i}(x))}\frac{1}{\omega(Q(y,\ell_{n}))}d\omega(y)\\
  &\lesssim_{C_{1},C_{4},C_{6}} \sum\limits_{i=1}^{m(x)}\frac{\omega(Q_{i}(x))}{\omega(Q_{i}(x))}c_{0}\mathcal{A}e^{-\frac{c_{1}\mathcal{A}}{60C_{5}}}\\
  &\lesssim_{C_{1},C_{4},C_{6}}c_{0}\mathcal{A}^{4N+1}e^{-\frac{c_{1}\mathcal{A}}{60C_{5}}},
\end{align*}
where in the penultimate inequality, we used the fact that the properties of dyadic cubes imply
\begin{equation*}
\omega(Q(y,\ell_{n}))\thicksim \omega(Q_{i}(x)),\qquad y\in\mathcal{G}(Q_i(x)),
\end{equation*}
with comparability constants independent of $y,\ell_{n}$ and $Q_{i}(x)$.
\end{proof}

It follows from Lemma \ref{lem:U-V-estimate} that the functions $b_{n}$ and $f_{n+1}$ satisfy $\rm(c1)-\rm(c5)$ for sufficiently large $\mathcal{A}$, if we can prove $\rm(c2)$.

{\bf{Step IV:}} Verification of condition $\rm(c2)$ for $\{b_{n}\}_{n\in{\mathbb{N}}}$. Take a dyadic cube $Q\in\mathcal{W}_{n}$ and fix $j\in{\mathbb{N}}$, and let $\tilde{I}\subset Q$ be a dyadic subcube with $\ell(\tilde{I})=\ell_{n+j}$. Assume that $b_{n}(y)\neq0$ for some $y\in{Q}$. We next prove that if $b_{n}(y)\neq0$ for some $y\in \tilde{I}$, then $b_{n+j}(x)=0$ for any $x\in \tilde{I}$. Take
\begin{equation*}
  Q=\tilde{I}_{0}\supset \tilde{I}_{1}\supset\ldots\supset \tilde{I}_{j}=\tilde{I},
\end{equation*}
where $\tilde{I}_{i}\in\mathcal{D}$ with $\ell(\tilde{I}_{i})=\ell_{n+i}$, for $i\in\{1,2,\ldots,j\}$. First, if $b_{n}(y)\neq0$ for some $y\in\tilde{I}_{j}$, by the definition of $b_{n}$ and the identity $\mathcal{G}(Q)=\bigcup_{\tilde{I}\in{\Theta(Q)}}\tilde{I}$, we have $\tilde{I}_{i}\notin\Theta(Q)$, $i\in\{1,2,\ldots,j\}$. Then we see
\begin{equation}\label{eq:main-estimate-5}
  |\mathcal{M}_{\tilde{I}_{i}}(f)-\mathcal{M}_{Q}(f)|<\frac{\mathcal{A}}{30C_{5}},\qquad i=1,2,\ldots,j.
\end{equation}

By choosing $\mathcal{A}$ sufficiently large in Lemma \ref{lem:U-V-estimate}, we ensure that for any $x\in{\mathbb{R}^{N}}$,
\begin{equation*}
  |U_{n}(x)-V_{n}(x)|\leq \frac{\mathcal{A}}{20C_{5}}-C_{7}(1+\alpha\mathcal{A}).
\end{equation*}
It follows from this estimate and \eqref{eq:main-estimate1} that 
\begin{align*}
\left|\mathcal{M}_{Q}(f_{n})-\mathcal{M}_{\tilde{I}_{1}}(U_{n})\right|
&=\left|\mathcal{M}_{\tilde{I}_{1}}\left(\mathcal{M}_{Q}(f_{n})-U_{n}\right)\right|\\
&\leq\left|\mathcal{M}_{\tilde{I}_{1}}\left(\mathcal{M}_{Q}(f_{n})-V_{n}\right)\right|+\left|\mathcal{M}_{\tilde{I}_{1}}\left(V_{n}-U_{n}\right)\right|
\leq\frac{(9C_{5}+7)\mathcal{A}}{20C_{5}}.
\end{align*}
Moreover, by \eqref{eq:estimate-Ui}, we obtain that there exists a constant $C_{8}>0$ depending only on $\mathbf{N}$ and independent of $\mathcal{A}$ such that for every $k\in\{1,2,\ldots,j\}$,
\begin{equation}\label{eq:main-estimate-c8}
\left|\mathcal{M}_{\tilde{I}_{k}}\left(\sum\limits_{i=1}^{n-1}U_{i}\right)-\mathcal{M}_{Q}\left(\sum\limits_{i=1}^{n-1}U_{i}\right)\right|
\leq C_{8}\mathcal{A}2^{-\mathcal{A}}.
\end{equation}
Then by these facts and \eqref{eq:main-estimate-5},
\begin{align*}
  \left|\mathcal{M}_{\tilde{I}_{1}}(f_{n+1})\right|
  &\leq\left|\mathcal{M}_{\tilde{I}_{1}}(f_{n})-\mathcal{M}_{\tilde{I}_{1}}(U_{n})\right|\\
  &\leq\left|\mathcal{M}_{\tilde{I}_{1}}(f_{n})-\mathcal{M}_{Q}(f_{n})\right|+\left|\mathcal{M}_{Q}(f_{n})-\mathcal{M}_{\tilde{I}_{1}}(U_{n})\right|\\
  &\leq|\mathcal{M}_{\tilde{I}_{1}}(f)-\mathcal{M}_{Q}(f)|+\left|\mathcal{M}_{Q}(f_{n})-\mathcal{M}_{\tilde{I}_{1}}(U_{n})\right|+\left|\mathcal{M}_{\tilde{I}_{1}}\left(\sum\limits_{i=1}^{n-1}U_{i}\right)-\mathcal{M}_{Q}\left(\sum\limits_{i=1}^{n-1}U_{i}\right)\right|\\
  &\leq \frac{1}{30C_{5}}\mathcal{A}+\frac{(9C_{5}+7)\mathcal{A}}{20C_{5}}+C_{8}\mathcal{A}2^{-\mathcal{A}}.
\end{align*}
Taking $\mathcal{A}$ sufficiently large at the outset so that $60C_{5}C_{8}<2^{\mathcal{A}}$, we have
\begin{equation*}
 \left|\mathcal{M}_{\tilde{I}_{1}}(f_{n+1})\right|<\frac{(9C_{5}+8)\mathcal{A}}{20C_{5}}.
\end{equation*}
It implies that $\rm(c4)$ holds for $n+1$ and $\tilde{I}_{1}$ so that 
\begin{equation*}
  b_{n+1}=U_{n+1}=0\quad \mathrm{on}\quad \tilde{I}_{1}.
\end{equation*}
Hence, we have $f_{n+2}(y)=f_{n+1}(y)$ for $y\in{\tilde{I}_{1}}$. Since $\tilde{I}_{2}\subset \tilde{I}_{1}$, using \eqref{eq:main-estimate-5} and \eqref{eq:main-estimate-c8} again, we have
\begin{equation*}
\left|\mathcal{M}_{\tilde{I}_{2}}(f_{n+2})\right|=\left|\mathcal{M}_{\tilde{I}_{2}}(f_{n+1})\right|\leq \frac{(9C_{5}+8)\mathcal{A}}{20C_{5}}. 
\end{equation*}
Thus by $\rm(c4)$, we have 
\begin{equation*}
  b_{n+2}=U_{n+2}=0\quad \mathrm{on}\quad \tilde{I}_{2}.
\end{equation*}
Repeating the argument for $\tilde{I}_{i}$ with $i=3,\ldots,j$, we have 
\begin{equation*}
  b_{n+i}=U_{n+i}=0\quad \mathrm{on}\quad \tilde{I}_{i}.
\end{equation*}
Then we conclude that $\rm(c2)$ holds, and the proof of Proposition \ref{pro:function-bn} is complete.
\end{proof}

Now, let us continue the proof of Theorem \ref{thm-BMO-onesetp}. By Proposition \ref{pro:function-bn} and the Lebesgue differentiation theorem,
\begin{equation*}
  \left|f(y)-\sum\limits_{i=1}^{\infty}U_{i}(y)\right|\leq\mathcal{A},\qquad \mathrm{a.e.}\,\,y\in{\mathbb{R}^{N}}.
\end{equation*}
For almost every $x\in{\mathbb{R}^{N}}$, we decompose $f$ as follows: 
\begin{align}\label{eq:f-addh0}
  f(x)\notag
  &=\sum\limits_{i=1}^{\infty}U_{i}(x)+f(x)-\sum\limits_{i=1}^{\infty}U_{i}(x)\\\notag
  &=\sum\limits_{i=1}^{\infty}\int_{\mathbb{R}^{N}}p_{\ell_{i}}(x,y)b_{i}(y)d\omega(y)+f(x)-\sum\limits_{i=1}^{\infty}U_{i}(x)+\sum\limits_{i=1}^{\infty}\int_{\mathbb{R}^{N}}\left(K_{\ell_{i}}(x,y)-p_{\ell_{i}}(x,y)\right)b_{i}(y)d\omega(y)\\
  &=\sum\limits_{i=1}^{\infty}\int_{\mathbb{R}^{N}}\mathbf{p}_{\ell_{i}}(x,y)b_{i}(y)d\omega(y)+h(0)+f(x)-\sum\limits_{i=1}^{\infty}U_{i}(x)+\psi^{(1)}(x),
\end{align}
where 
\begin{equation*}
  h(0):=\sum\limits_{i=1}^{\infty}\int_{\mathbb{R}^{N}}p_{\ell_{i}}(0,y)\cdot \mathbf{1}_{\{\ell_{i}>1\}}(y,\ell_{i})b_{i}(y)d\omega(y),\quad \psi^{(1)}(x):=\sum\limits_{i=1}^{\infty}\int_{\mathbb{R}^{N}}\left(K_{\ell_{i}}(x,y)-p_{\ell_{i}}(x,y)\right)b_{i}(y)d\omega(y).
\end{equation*}

Using \eqref{eq:p-G-invariance}, $\rm(c1)$ and $\rm(c2)$ in Proposition \ref{pro:function-bn}, we have
\begin{align}\label{eq:h(0)-estimate}
  |h(0)|
  &\leq C_{6}\mathcal{A}\sum\limits_{i=1}^{\left\lfloor \frac{m_{0}}{\mathcal{A}}\right\rfloor}\int_{\supp b_{i}}p_{\ell_{i}}(0,y)\cdot \mathbf{1}_{\{\ell_{i}>1\}}(y,\ell_{i})d\omega(y)\\\notag
  &\leq C_{6}\mathcal{A}\sum\limits_{i=1}^{\left\lfloor \frac{m_{0}}{\mathcal{A}}\right\rfloor}\int_{\mathbb{R}^{N}}p_{\ell_{i}}(0,y)d\omega(y)\\\notag
  &\leq C_{6}m_{0}.\notag
\end{align}
Let us denote 
\begin{equation*}
  b^{(1)}:=\sum\limits_{i=1}^{\infty}b_{i}, 
\end{equation*}
and define
\begin{eqnarray*} 
t_{1}(y):=
\left\{\begin{array}{ll}
\ell_{i},\quad &y\in \supp b_{i}\,\,\,\mathrm{for }\,\,\, i\in{\mathbb{N}};\\
\,\\
1,\quad &y\in{\mathbb{R}^{N}\backslash\bigcup_{i\in{\mathbb{N}}}\supp b_{i}}.
\end{array}\right.
\end{eqnarray*}
Obviously, $0<t_{1}(y)\leq\max\{1,2^{m_{0}-\mathcal{A}}\}$, for $y\in{\mathbb{R}^{N}}$. By $\rm(c2)$ in Proposition \ref{pro:function-bn}, we write
\begin{equation*}
\psi^{(1)}(x)=\int_{Q(0,2^{m_{0}})}(K-p)_{t_{1}(y)}(x,y)b^{(1)}(y)d\omega(y),\qquad x\in{\mathbb{R}^{N}}.
\end{equation*}
Note that by $\rm(c1)$ and $\rm(c2)$ in Proposition \ref{pro:function-bn}, we have $|b^{(1)}(y)|\leq C_{6}\mathcal{A}$, for $y\in{Q(0,2^{m_{0}})}$. We now establish $\psi^{(1)}\in{\mathrm{BMO}(\mathbb{R}^{N}, d\omega)}$ and derive the following useful estimate.
\begin{lemma}\label{lem:psi-estimate}
There exists a constant $C_{9}>0$, independent of $\mathcal{A}$, such that
\begin{equation}\label{eq:main-estimate-6}
\left\|\psi^{(1)}\right\|_{\ast,\kappa}
\leq C_{9}\mathcal{A}^{-2}<\frac{1}{4}
\end{equation}
whenever $\sqrt{4C_{9}}<\mathcal{A}$.
\end{lemma}
\begin{proof} 
Suppose $B:=B(x_{0},r)$ with $x_{0}\in{\mathbb{R}^{N}}$ and $r>0$. Write
\begin{align*}
  &\frac{1}{\omega(B)}\int_{B}\left|\psi^{(1)}(x)-\mathcal{M}_{B}(\psi^{(1)})\right|d\omega(x)\\
  &\quad\leq\frac{1}{\omega(B)^{2}}\int_{B}\int_{B}\int_{Q(0,2^{m_{0}})}\left|(K-p)_{t_{1}(y)}(x,y)-(K-p)_{t_{1}(y)}(z,y)\right||b^{(1)}(y)|d\omega(y)d\omega(z)d\omega(x)\\
  &\quad\leq\frac{C_{6}\mathcal{A}}{\omega(B)^{2}}\int_{B}\int_{B}\left(\int_{\{d(x_{0},y)\leq5r\}\cap Q(0,2^{m_{0}})}+\int_{Q(0,2^{m_{0}})\backslash \{d(x_{0},y)\leq5r\}}\right)\Big|(K-p)_{t_{1}(y)}(x,y)\\
  &\quad\quad\quad\quad\quad\quad\quad\quad\quad\quad\quad\quad\quad-(K-p)_{t_{1}(y)}(z,y)\Big|d\omega(y)d\omega(z)d\omega(x)\\
  &\quad=:\mathcal{L}_{1}+\mathcal{L}_{2},
\end{align*}
where $\{d(x_{0},y)\leq5r\}:=\{y\in{\mathbb{R}^{N}}:\,d(x_{0},y)\leq5r\}$.

Consider $\mathcal{L}_{1}$. Write
\begin{align*}
 \mathcal{L}_{1}
  &\leq\frac{C_{6}\mathcal{A}}{\omega(B)}\bigg(\int_{\{d(x_{0},y)\leq5r\}\cap Q(0,2^{m_{0}})}\int_{B}\left|(K-p)_{t_{1}(y)}(x,y)\right|d\omega(x)d\omega(y)\\
  &\quad\quad\quad\quad\quad+\int_{\{d(x_{0},y)\leq5r\}\cap Q(0,2^{m_{0}})}\int_{B}\left|(K-p)_{t_{1}(y)}(z,y)\right|d\omega(z)d\omega(y)\bigg).
\end{align*}
If we show that for any $y\in{\{d(x_{0},y)\leq5r\}\cap Q(0,2^{m_{0}})}$,
\begin{equation}\label{eq:estimate-L1}
\int_{B}\left|(K-p)_{t_{1}(y)}(x,y)\right|d\omega(x)\lesssim_{C_{1},C_{4}}\mathcal{A}^{-3},
\end{equation}
then 
\begin{equation*}
\mathcal{L}_{1}\lesssim_{C_{1},C_{4},C_{6}}\mathcal{A}^{-2},
\end{equation*}
where the implicit constants are independent of $B$ and $y$.

To prove \eqref{eq:estimate-L1}, by the definition of $K$, we have that for $x\in{B}$ and $y\in{\{d(x_{0},y)\leq5r\}\cap Q(0,2^{m_{0}})}$,
\begin{eqnarray*} 
(K-p)_{t_{1}(y)}(x,y)=
\left\{\begin{array}{ll}
p_{t_{1}(y)}(x,y)\left(\frac{1}{\phi_{t_{1}(y)}^{\mathcal{A}}(x)}-1\right),\quad &\|x-y\|\leq t_{1}(y)\mathcal{A}^{4};\\
\,\\
(K-p)_{t_{1}(y)}(x,y),\quad &t_{1}(y)\mathcal{A}^{4}<\|x-y\|<2t_{1}(y)\mathcal{A}^{4};\\
\,\\
-p_{t_{1}(y)}(x,y),\quad &\|x-y\|\geq2t_{1}(y)\mathcal{A}^{4}.
\end{array}\right.
\end{eqnarray*}
Using \eqref{eq:poisson-upper}, $\rm{(a2)}$ in Lemma \ref{lem-kernel-properties}, and a computation analogous to \eqref{eq:sequence-Dunkl}, we obtain 
\begin{align}\label{eq:phi-approxmate}
  \left|\frac{1}{\phi_{t_{1}(y)}^{\mathcal{A}}(x)}-1\right|
  &\leq\frac{1}{C_{4}}\int_{\mathbb{R}^{N}}p_{t_{1}(y)}(x,z)\left|1-\chi_{t_{1}(y)}^{\mathcal{A}}(x,z)\right|d\omega(z)\\\notag
  &\leq\frac{C_{1}}{C_{4}}\int_{\|x-z\|\geq t_{1}(y)\mathcal{A}^{4}}\frac{1}{\omega(B(x, t_{1}(y)+d(x,z)))}\frac{t_{1}(y)}{t_{1}(y)+\|x-z\|}d\omega(z)\\\notag
  &\leq\frac{C_{1}}{C_{4}\mathcal{A}^{3}}\int_{\mathbb{R}^{N}}\frac{1}{\omega(B(x, t_{1}(y)+d(x,z)))}\left(\frac{t_{1}(y)}{t_{1}(y)+\|x-z\|}\right)^{1/4}d\omega(z)\\\notag
  &\leq\frac{C_{1}}{C_{4}\mathcal{A}^{3}}\left(\int_{d(x,z)\leq t_{1}(y)}+\sum_{i=0}^{\infty}\frac{1}{2^{i/4}}\int_{d(x,z)\thicksim2^{i}t_{1}(y)}\right)\frac{1}{\omega(B(x, t_{1}(y)+d(x,z)))}d\omega(z)\\\notag
  &\lesssim_{C_{1},C_{4}}\mathcal{A}^{-3}.\notag
\end{align}
Then for $\|x-y\|\leq t_{1}(y)\mathcal{A}^{4}$, 
\begin{equation*}
\left|(K-p)_{t_{1}(y)}(x,y)\right|\lesssim_{C_{1},C_{4}}\mathcal{A}^{-3}p_{t_{1}(y)}(x,y).
\end{equation*}

In the remaining cases, namely $t_{1}(y)\mathcal{A}^{4}<\|x-y\|<2t_{1}(y)\mathcal{A}^{4}$ or $\|x-y\|\geq2t_{1}(y)\mathcal{A}^{4}$, we have
\begin{align*}
\left|(K-p)_{t_{1}(y)}(x,y)\right|
&\leq\frac{1}{C_{4}}p_{t_{1}(y)}(x,y)\leq\frac{C_{1}}{C_{4}}\frac{1}{\omega(B(y, t_{1}(y)+d(x,y)))}\frac{t_{1}(y)}{t_{1}(y)+\|x-y\|}\\
&\lesssim_{C_{1},C_{4}}\mathcal{A}^{-3}\frac{1}{\omega(B(y, t_{1}(y)+d(x,y)))}\left(\frac{t_{1}(y)}{t_{1}(y)+\|x-y\|}\right)^{1/4}.
\end{align*}
It follows from these facts that 
\begin{align*}
  &\int_{B}\left|(K-p)_{t_{1}(y)}(x,y)\right|d\omega(x)\\
  &\quad\leq\left(\int_{\|x-y\|\leq t_{1}(y)\mathcal{A}^{4}}+\int_{t_{1}(y)\mathcal{A}^{4}<\|x-y\|<2t_{1}(y)\mathcal{A}^{4}}+\int_{\|x-y\|\geq2t_{1}(y)\mathcal{A}^{4}}\right)\left|(K-p)_{t_{1}(y)}(x,y)\right|d\omega(x)\\
  &\quad\lesssim_{C_{1},C_{4}}\mathcal{A}^{-3}\int_{\mathbb{R}^{N}}p_{t_{1}(y)}(x,y)d\omega(x)+\mathcal{A}^{-3}\int_{\mathbb{R}^{N}}\frac{1}{\omega(B(y, t_{1}(y)+d(x,y)))}\left(\frac{t_{1}(y)}{t_{1}(y)+\|x-y\|}\right)^{1/4}d\omega(x)\\
  &\quad\lesssim_{C_{1},C_{4}}\mathcal{A}^{-3}.
\end{align*}
It shows \eqref{eq:estimate-L1}.

As for $\mathcal{L}_{2}$, for $x,z\in{B}$ and $y\in{Q(0,2^{m_{0}})\backslash \{d(x_{0},y)\leq5r\}}$, we have
\begin{equation*}
  \|x-z\|\leq2r\leq \min\{d(y,x),d(y,x_{0}),d(y,z)\}/2,
\end{equation*}
and 
\begin{equation*}
  \|y-x\|\thicksim\|y-x_{0}\|\thicksim\|y-z\|,\quad d(x,y)\thicksim d(x_{0},y)\thicksim d(z,y),
\end{equation*}
and 
\begin{equation*}
  \omega(B(x,d(x,y)))\thicksim \omega(B(x_{0},d(x_{0},y)))\thicksim \omega(B(z,d(z,y))).
\end{equation*}
To estimate the kernel in $\mathcal{L}_{2}$, we consider the following cases:

Case 1: $\|x-y\|\leq t_{1}(y)\mathcal{A}^{4}$ and $\|y-z\|\leq t_{1}(y)\mathcal{A}^{4}$. Observe that
\begin{equation*}
4r\leq t_{1}(y)\mathcal{A}^{4}.
\end{equation*}
Using the estimates \eqref{eq:phi-approxmate} and $\rm{(a2)}$ in Lemma \ref{lem-kernel-properties},
\begin{align*}
&\left|(K-p)_{t_{1}(y)}(x,y)-(K-p)_{t_{1}(y)}(z,y)\right|\\
&\quad\leq\left|\left(p_{t_{1}(y)}(x,y)-p_{t_{1}(y)}(z,y)\right)\left(\frac{1}{\phi_{t_{1}(y)}^{\mathcal{A}}(x)}-1\right)\right|+\left|p_{t_{1}(y)}(z,y)\left(\frac{1}{\phi_{t_{1}(y)}^{\mathcal{A}}(x)}-\frac{1}{\phi_{t_{1}(y)}^{\mathcal{A}}(z)}\right)\right|\\
&\quad\lesssim_{C_{1},C_{3},C_{4}}\frac{\mathcal{A}^{-3}}{\omega(B(x,d(x,y)))}\frac{\|x-z\|}{\|x-y\|}+\frac{\mathcal{A}^{-4}}{\omega(B(z,d(z,y)))}\frac{\|x-z\|}{\|z-y\|}.
\end{align*}

Case 2: $\|x-y\|\leq t_{1}(y)\mathcal{A}^{4}$ and $\|y-z\|\geq t_{1}(y)\mathcal{A}^{4}$. Under the assumptions of this case, together with $4r\leq t_{1}(y)\mathcal{A}^{4}$ and 
\begin{equation*}
\|y-z\|\leq\|x-y\|+\|x-z\|,
\end{equation*}
we obtain
\begin{equation*}
t_{1}(y)\mathcal{A}^{4}/2\leq\|x-y\|\leq t_{1}(y)\mathcal{A}^{4}\leq\|y-z\|\leq\frac{3}{2}t_{1}(y)\mathcal{A}^{4}.
\end{equation*}
It follows from these comparability estimates, \eqref{eq:poisson-Holder} and $\rm{(a3)}$ in Lemma \ref{lem-kernel-properties} that
\begin{align}\label{eq:k-p-estimate}
\left|(K-p)_{t_{1}(y)}(x,y)-(K-p)_{t_{1}(y)}(z,y)\right|
&\leq\left|K_{t_{1}(y)}(x,y)-K_{t_{1}(y)}(z,y)\right|+\left|p_{t_{1}(y)}(x,y)-p_{t_{1}(y)}(z,y)\right|\\\notag
&\lesssim_{C_{1},C_{2},C_{3},C_{4}}\frac{\mathcal{A}^{-3}\|x-z\|}{\omega(B(x,d(x,y)))}\left(\frac{1}{\|x-y\|}+\frac{1}{d(x,y)}\right).\notag
\end{align}

Case 3: $\|x-y\|\geq t_{1}(y)\mathcal{A}^{4}$ and $\|y-z\|\leq t_{1}(y)\mathcal{A}^{4}$. By proceeding as in Case 2, we also obtain \eqref{eq:k-p-estimate}.

Case 4: $\|x-y\|\geq t_{1}(y)\mathcal{A}^{4}$ and $\|y-z\|\geq t_{1}(y)\mathcal{A}^{4}$. Using \eqref{eq:poisson-Holder} and $\rm(a3)$ in Lemma \ref{lem-kernel-properties} again, \eqref{eq:k-p-estimate} also holds.

From these estimates in Cases 1-4, we get
\begin{align*}
\mathcal{L}_{2}
&=\frac{C_{6}\mathcal{A}}{\omega(B)^{2}}\int_{B}\int_{B}\int_{Q(0,2^{m_{0}})\backslash \{d(x_{0},y)\leq5r\}}\Big|(K-p)_{t_{1}(y)}(x,y)-(K-p)_{t_{1}(y)}(z,y)\Big|d\omega(y)d\omega(z)d\omega(x)\\
&\lesssim_{C_{1},C_{2},C_{3},C_{4},C_{6}}\mathcal{A}^{-2}.
\end{align*}
Combining the estimates for $\mathcal{L}_{1}$ and $\mathcal{L}_{2}$ and absorbing all implicit constants into $C_{9}$, we obtain \eqref{eq:main-estimate-6}.
\end{proof}

Notice that the function $\psi^{(1)}$ does not have compact support. Nevertheless, by \eqref{eq:poisson-upper}, we have
\begin{align}\label{eq:main-estimate-7}
  \left|\psi^{(1)}(x)\right|
  &\leq\int_{Q(0,2^{m_{0}})}\left|(K-p)_{t_{1}(y)}(x,y)\right||b^{(1)}(y)|d\omega(y)\\\notag
  &\lesssim\int_{Q(0,2^{m_{0}})}\frac{t_{1}(y)}{\big(t_{1}(y)+\big|\|x\|-\|y\|\big|\big)^{\mathbf{N}+1}}|b^{(1)}(y)|d\omega(y)\\\notag
  &\lesssim \frac{C_{6}\mathcal{A}2^{m_{0}\mathbf{N}}}{\|x\|^{\mathbf{N}}},\notag
\end{align}
for $\|x\|>2^{m_{0}}\sqrt{N}$. We therefore truncate $\psi^{(1)}$ outside a sufficiently large ball.
More precisely, set
\begin{eqnarray*} 
\widetilde{\psi}^{(1)}(x):=
\left\{\begin{array}{ll}
\psi^{(1)}(x),\quad &\|x\|\leq2^{m_{0}-1+n_{0}\mathcal{A}};\\
\,\\
0,\quad &\|x\|>2^{m_{0}-1+n_{0}\mathcal{A}},
\end{array}\right.
\end{eqnarray*}
where $n_{0}\in\mathbb{N}$ will be chosen sufficiently large below.
Assume that $2^{\mathcal{A}}\geq\sqrt{N}$. It follows from \eqref{eq:main-estimate-7} that for $\|x\|>2^{m_{0}-1+n_{0}\mathcal{A}}$,
\begin{equation*}
\left|\psi^{(1)}(x)\right|\lesssim\frac{C_{6}\mathcal{A}2^{m_{0}\mathbf{N}}}{2^{(m_{0}-1+n_{0}\mathcal{A})\mathbf{N}}}\lesssim C_{6}\mathcal{A}/2^{n_{0}\mathcal{A}\mathbf{N}}.
\end{equation*}
Then there exists a sufficiently large $n_{0}\in{\mathbb{N}}$ such that
\begin{equation}\label{eq:eq:main-estimate-8}
\left\|\widetilde{\psi}^{(1)}-\psi^{(1)}\right\|_{\ast,\kappa}\lesssim C_{6}\mathcal{A}/2^{n_{0}\mathcal{A}\mathbf{N}}<\frac{1}{6}.
\end{equation}
This estimate together with \eqref{eq:main-estimate-6} yields that $\|\widetilde{\psi}^{(1)}\|_{\ast,\kappa}<1/2$. Moreover, using \eqref{eq:main-estimate-7} again, we have 
\begin{equation*}
\left\|\widetilde{\psi}^{(1)}-\psi^{(1)}\right\|_{L^{\infty}(\mathbb{R}^{N},d\omega)}\leq\mathcal{A}.
\end{equation*}

Returning to the decomposition of $f$, for $x\in{\mathbb{R}^{N}}$,
\begin{align*}
  f(x)
  &=\sum\limits_{i=1}^{\infty}\int_{\mathbb{R}^{N}}\mathbf{p}_{\ell_{i}}(x,y)b_{i}(y)d\omega(y)+h(0)+f(x)-\sum\limits_{i=1}^{\infty}U_{i}(x)+\left(\psi^{(1)}(x)-\widetilde{\psi}^{(1)}(x)\right)+\widetilde{\psi}^{(1)}(x)\\
  &=\sum\limits_{i=1}^{\infty}\int_{\mathbb{R}^{N}}\mathbf{p}_{\ell^{(1)}_{i}}(x,y)b_{i}^{(1)}(y)d\omega(y)+h^{(1)}(0)+\tilde{b}_{0}^{(1)}(x)+\widetilde{\psi}^{(1)}(x),
\end{align*}
where for convenience, we denote
\begin{equation*}
\ell^{(1)}_{i}:=\ell_{i},\quad
b_{i}^{(1)}:=b_{i},\quad 
h^{(1)}(0):=h(0)\quad\mathrm{and} \quad 
\tilde{b}_{0}^{(1)}:=f-\sum\limits_{i=1}^{\infty}U_{i}+\psi^{(1)}-\widetilde{\psi}^{(1)}.
\end{equation*}
Notice that $\supp b_{i}^{(1)}\cap\supp b_{j}^{(1)}=\emptyset$ for all $i,j\geq1$ and $i\neq j$, and 
\begin{equation*}
  \left|b_{i}^{(1)}(x)\right|\leq C_{6}\mathcal{A}\quad \mathrm{and} \quad \left|\tilde{b}_{0}^{(1)}(x)\right|\leq2\mathcal{A},\quad x\in{\mathbb{R}^{N}},
\end{equation*}
and 
\begin{equation*}
|h^{(1)}(0)|\leq C_{6}m_{0},\quad \supp \widetilde{\psi}^{(1)}\subset B\left(0,2^{m_{0}-1+n_{0}\mathcal{A}}\right)\quad \mathrm{and} \quad \left\|\widetilde{\psi}^{(1)}\right\|_{\ast,\kappa}<\frac{1}{2}.
\end{equation*}

Fix $n_{0}\in{\mathbb{N}}$ such that \eqref{eq:eq:main-estimate-8} holds. For the function $2\widetilde{\psi}^{(1)}$, following a similar argument as for $f$, we have
\begin{equation*}
2\widetilde{\psi}^{(1)}(x)=\sum\limits_{i=1}^{\infty}\int_{\mathbb{R}^{N}}\mathbf{p}_{\ell^{(2)}_{i}}(x,y)b_{i}^{(2)'}(y)d\omega(y)+h^{(2)'}(0)+\tilde{b}_{0}^{(2)'}(x)+\widetilde{\psi}^{(2)'}(x),\quad x\in{\mathbb{R}^{N}}.
\end{equation*}
Let us denote 
\begin{equation*}
b_{i}^{(2)}=\frac{b_{i}^{(2)'}}{2},\quad 
h^{(2)}(0)=\frac{h^{(2)'}(0)}{2},\quad 
\tilde{b}_{0}^{(2)}:=\frac{\tilde{b}_{0}^{(2)'}}{2},\quad \widetilde{\psi}^{(2)}:=\frac{\widetilde{\psi}^{(2)'}}{2}.
\end{equation*}
Then
\begin{equation*}
\widetilde{\psi}^{(1)}(x)=\sum\limits_{i=1}^{\infty}\int_{\mathbb{R}^{N}}\mathbf{p}_{\ell^{(2)}_{i}}(x,y)b_{i}^{(2)}(y)d\omega(y)+h^{(2)}(0)+\tilde{b}_{0}^{(2)}(x)+\widetilde{\psi}^{(2)}(x),\qquad x\in{\mathbb{R}^{N}},
\end{equation*}
where 
\begin{equation*}
  \ell^{(2)}_{i}:=2^{m_{0}-(i-n_{0})\mathcal{A}},
\end{equation*}
and
$\supp b_{i}^{(2)}\cap\supp b_{j}^{(2)}=\emptyset$ for all $i,j\geq1$ and $i\neq j$, and 
\begin{equation*}
  \left|b_{i}^{(2)}(x)\right|\leq \frac{C_{6}\mathcal{A}}{2}\quad \mathrm{and} \quad \left|\tilde{b}_{0}^{(2)}(x)\right|\leq\frac{2\mathcal{A}}{2},\qquad\text{for any}\quad x\in{\mathbb{R}^{N}},
\end{equation*}
and 
\begin{equation*}
\supp \widetilde{\psi}^{(2)}\subset B\left(0,2^{m_{0}-1+2n_{0}\mathcal{A}}\right)\quad \mathrm{and} \quad \left\|\widetilde{\psi}^{(2)}\right\|_{\ast,\kappa}<\frac{1}{2^{2}}.
\end{equation*}
Moreover, by an argument analogous to \eqref{eq:h(0)-estimate}, we have
\begin{equation*}
  |h^{(2)}(0)|\leq \frac{C_{6}(m_{0}+n_{0}\mathcal{A})}{2}.
\end{equation*}

Repeating this argument, we get
\begin{equation*}
f(x)=\sum\limits_{j=1}^{\infty}\left(\sum\limits_{i=1}^{\infty}\int_{\mathbb{R}^{N}}\mathbf{p}_{\ell^{(j)}_{i}}(x,y)b_{i}^{(j)}(y)d\omega(y)+h^{(j)}(0)+\tilde{b}_{0}^{(j)}(x)\right)\qquad \mathrm{in}\quad\mathrm{BMO}(\mathbb{R}^{N}, d\omega),
\end{equation*}
where  
\begin{equation*}
  \ell^{(j)}_{i}:=2^{m_{0}-(i-(j-1)n_{0})\mathcal{A}},\qquad |h^{(j)}(0)|\leq2^{-j+1}C_{6}(m_{0}+n_{0}(j-1)\mathcal{A})
\end{equation*}
and
\begin{equation}\label{eq:b-estimates}
\supp b_{i}^{(j)}\subset Q(0,2^{m_{0}+jn_{0}\mathcal{A}}),\quad
\left\|\tilde{b}_{0}^{(j)}\right\|_{L^{\infty}(\mathbb{R}^{N}, d\omega)}\leq2^{-j+2}\mathcal{A},\quad \left\|b_{i}^{(j)}\right\|_{L^{\infty}(\mathbb{R}^{N}, d\omega)}\leq2^{-j+1}C_{6}\mathcal{A},
\end{equation}
And for each fixed $j\in\mathbb{N}$, the supports $\supp b_i^{(j)}$, $i\in\mathbb{N}$ are pairwise disjoint. Moreover, the estimate of $h^{(j)}(0)$ gives
\begin{equation*}
\sum\limits_{j=1}^{\infty}\left|h^{(j)}(0)\right|\leq\sum\limits_{j=1}^{\infty}\frac{C_{6}(m_{0}+n_{0}(j-1)\mathcal{A})}{2^{j-1}}\leq 2C_{6}(m_{0}+n_{0}\mathcal{A})<\infty.
\end{equation*}

For notational simplicity, for each $j\in{\mathbb{N}}$, we replace the index $(i-(j-1)n_{0})$ by $i$ in $\ell^{(j)}_{i}$, and correspondingly relabel $b_{i}^{(j)}$ as $b_{i-(j-1)n_{0}}^{(j)}$. With this convention, and since $\sum_{j=1}^{\infty}h^{(j)}(0)$ is a fixed constant, we can write
\begin{equation*}
f(x)=\sum\limits_{j=1}^{\infty}\left(\sum\limits_{i\in{\mathbb{Z}}}\int_{\mathbb{R}^{N}}\mathbf{p}_{\ell_{i}}(x,y)b_{i}^{(j)}(y)d\omega(y)+\tilde{b}_{0}^{(j)}(x)\right)\qquad \mathrm{in}\quad\mathrm{BMO}(\mathbb{R}^{N}, d\omega),
\end{equation*}
where
\begin{equation*}
b_{i}^{(j)}=0,\qquad \mathrm{for}\,\,\,i\leq -(j-1)n_{0}.
\end{equation*}
It follows from the estimates in \eqref{eq:b-estimates} that
\begin{equation*}
\left\|\sum\limits_{j=1}^{\infty}\tilde{b}_{0}^{(j)}\right\|_{L^{\infty}(\mathbb{R}^{N}, d\omega)}\leq 4\mathcal{A},
\end{equation*}
and using \eqref{eq:p-G-invariance},
\begin{align*}
  \left\|\sum\limits_{j=1}^{\infty}\sum\limits_{i\leq0}\int_{\mathbb{R}^{N}}\left|\mathbf{p}_{\ell_{i}}(\cdot,y)b_{i}^{(j)}(y)\right|d\omega(y)\right\|_{L^{\infty}(\mathbb{R}^{N}, d\omega)}
  &=\left\|\sum\limits_{j=1}^{\infty}\sum\limits_{i>-(j-1)n_{0}}^{0}\int_{\mathbb{R}^{N}}\left|\mathbf{p}_{\ell_{i}}(\cdot,y)b_{i}^{(j)}(y)\right|d\omega(y)\right\|_{L^{\infty}(\mathbb{R}^{N}, d\omega)}\\
  &\leq C_{6}\sum\limits_{j=1}^{\infty}\frac{((j-1)n_{0}+1)\mathcal{A}}{2^{j-1}}\\
  &\leq C_{6}(2n_{0}+2)\mathcal{A}.
\end{align*}
Furthermore, for any fixed $R>1$, it follows from \eqref{eq:poisson-upper} that
\begin{align*}
  &\int_{B(0,R)}\sum\limits_{j=1}^{\infty}\sum\limits_{i=1}^{\infty}\int_{\mathbb{R}^{N}}\left|\mathbf{p}_{\ell_{i}}(x,y)b_{i}^{(j)}(y)\right|d\omega(y)d\omega(x)\\
  &\quad\lesssim\sum\limits_{j=1}^{\infty}\sum\limits_{i=1}^{\infty}\int_{B(0,2R)}\left(\int_{\mathbb{R}^{N}}p_{\ell_{i}}(x,y)d\omega(x)+\omega(B(0,R))\right)\left|b_{i}^{(j)}(y)\right|d\omega(y)\\
  &\qquad+\sum\limits_{j=1}^{\infty}\sum\limits_{i=1}^{\infty}\int_{\mathbb{R}^{N}\backslash B(0,2R)}\int_{B(0,R)}\left(p_{\ell_{i}}(x,y)+p_{\ell_{i}}(0,y)\right)d\omega(x)\left|b_{i}^{(j)}(y)\right|d\omega(y)\\
  &\quad\lesssim_{R,C_{1},C_{6},\ell_{1}}\mathcal{A}+\mathcal{A}\int_{\mathbb{R}^{N}\backslash B(0,2R)}\frac{1}{\omega(B(0, d(0,y)))}\frac{1}{\|y\|}d\omega(y)<\infty.
\end{align*}
Then
\begin{equation*}
\sum\limits_{j=1}^{\infty}\sum\limits_{i=1}^{\infty}\int_{\mathbb{R}^{N}}\left|\mathbf{p}_{\ell_{i}}(x,y)b_{i}^{(j)}(y)\right|d\omega(y)<\infty,\qquad \mathrm{for}\quad\mathrm{a.e.}\,\,x\in{\mathbb{R}^{N}}.
\end{equation*}
Thus by changing the order of the summation, we get 
\begin{equation*}
f(x)
=\sum\limits_{i=1}^{\infty}\int_{\mathbb{R}^{N}}\mathbf{p}_{\ell_{i}}(x,y) \left(\sum\limits_{j=1}^{\infty}b_{i}^{(j)}(y)\right)d\omega(y)+\sum\limits_{j=1}^{\infty}\sum\limits_{i\leq0}\int_{\mathbb{R}^{N}}\mathbf{p}_{\ell_{i}}(x,y) b_{i}^{(j)}(y)d\omega(y)+\sum\limits_{j=1}^{\infty}\tilde{b}_{0}^{(j)}(x),
\end{equation*}
in $\mathrm{BMO}(\mathbb{R}^{N}, d\omega)$. Let us denote 
\begin{equation*}
  g_{i}:=\sum\limits_{j=1}^{\infty}b_{i}^{(j)},\quad i\in{\mathbb{N}},\qquad \mathrm{and} \quad g_{0}:=\sum\limits_{j=1}^{\infty}\left(\sum\limits_{i\leq0}\int_{\mathbb{R}^{N}}\mathbf{p}_{\ell_{i}}(\cdot,y) b_{i}^{(j)}(y)d\omega(y)+\tilde{b}_{0}^{(j)}\right).
\end{equation*}
Note that the constants $\mathcal{A}$ and $n_0$ are chosen independently of both $m_0$ and the function $f$. This completes the proof of Theorem \ref{thm-BMO-onesetp}.
\end{proof}

\section{The proofs of Theorems \ref{thm:balayage-measure} and \ref{thm-Hardy-balayage}}\label{sec:thm1.1-1.2}
This section is to present the proofs of Theorems \ref{thm:balayage-measure} and \ref{thm-Hardy-balayage} in Subsections \ref{sec:thm1.1} and \ref{sec:thm1.3}, respectively. 

\subsection{Proof of Theorem \ref{thm:balayage-measure}}\label{sec:thm1.1}
\hskip\parindent

First, we prove $\rm(i)$ in Theorem \ref{thm:balayage-measure}. For $E\subseteq \R^N$, define
$$
        \widehat{O}(E):=\bigcup_{\sigma\in G}\widehat{\sigma(E)}.
$$
\begin{proof}[Proof of $\rm(i)$ in Theorem \ref{thm:balayage-measure}]
Let  $\mu\in{\Re(\mathbb{R}^{N+1}_{+})}$. We first show that 
\begin{equation}\label{eq:local-integral}
\widetilde{S}_{\mu,\mathcal{P}}\in{L_{\mathrm{loc}}^{1}(\mathbb{R}^{N}, d\omega)}.
\end{equation}
We fix $R>1$ and write 
\begin{align*}
  &\int_{B(0,R)}\iint_{\mathbb{R}^{N+1}_{+}}\left|p_{t}(x,y)-p_{t}(0,y)\cdot \mathbf{1}_{\{t>1\}}(y,t)\right|d|\mu|(y,t)d\omega(x)\\
   &\quad\leq\int_{B(0,R)}\iint_{\mathbb{R}^{N}\times\{0<t\leq1\}}p_{t}(x,y)d|\mu|(y,t)d\omega(x)+\int_{B(0,R)}\iint_{\mathbb{R}^{N}\times\{t>1\}}\left|p_{t}(x,y)-p_{t}(0,y)\right|d|\mu|(y,t)d\omega(x)\\
   &\quad=: K_{1}+K_{2}.
\end{align*}

For $K_{1}$, by \eqref{eq:p-G-invariance} and \eqref{eq:poisson-upper},
\begin{align*}
  K_{1}
  &\lesssim\iint_{B(0,2R)\times\{0<t\leq1\}}\int_{\mathbb{R}^{N}}p_{t}(x,y)d\omega(x)d|\mu|(y,t)\\
  &\quad\quad+\int_{B(0,R)}\sum\limits_{i=1}^{\infty}\iint_{[B(0,2^{i+1}R)\backslash B(0,2^{i}R)]\times\{0<t\leq1\}}\frac{1}{\omega(B(x, t+d(x,y)))}\frac{t}{t+\|x-y\|}d|\mu|(y,t)d\omega(x)\\
  &\lesssim_{R}\|\mu\|_{\mathcal{C}}.
\end{align*}
As for $K_{2}$, using \eqref{eq:poisson-upper} and \eqref{eq:poisson-smooth},
\begin{align*}
  K_{2}
  &\lesssim\int_{B(0,R)}\iint_{\left\{(y,t)\in{\mathbb{R}^{N+1}_{+}}:\,\|y\|+t<4R,\,\,t>1\right\}}1d|\mu|(y,t)d\omega(x)\\
  &\quad\quad+\int_{B(0,R)}\sum\limits_{i=1}^{\infty}\iint_{\left\{(y,t)\in{\mathbb{R}^{N+1}_{+}}:\,4^{i}R\leq\|y\|+t<4^{i+1}R,\,\,t>1\right\}}\frac{\|x\|}{\omega(B(\xi, t+d(\xi,y)))}\frac{1}{t+\|\xi-y\|}d|\mu|(y,t)d\omega(x)\\
  &\lesssim_{R}\|\mu\|_{\mathcal{C}},
\end{align*}
where $\xi=sx$, for some $s\in{(0,1)}$. Combining the estimates for $K_{1}$ and $K_{2}$, we obtain \eqref{eq:local-integral}.

Now, we prove that
\begin{equation}\label{ep:i-mu}
  \left\|\widetilde{S}_{\mu,\mathcal{P}}\right\|_{\ast,\kappa}\lesssim\|\mu\|_{\mathcal{C}},
\end{equation}
where the implicit constant is independent of $\mu$.
Let $B:=B(x_{0}, r_{0})$ with $x_{0}\in{\mathbb{R}^{N}}$ and $r_{0}>0$. Write
\begin{align*}
  &\frac{1}{\omega(B)}\int_{B}\left|\widetilde{S}_{\mu,\mathcal{P}}(x)-\mathcal{M}_{B}\left(\widetilde{S}_{\mu,\mathcal{P}}\right)\right|d\omega(x)\\
  &\quad\leq\frac{1}{\big(\omega(B)\big)^{2}}\int_{B}\int_{B}\iint_{\mathbb{R}^{N+1}_{+}}\left|p_{t}(x,y)-p_{t}(z,y)\right|d|\mu|(y,t)d\omega(z)d\omega(x)\\
  &\quad\leq\frac{1}{\big(\omega(B)\big)^{2}}\int_{B}\int_{B}\left(\iint_{\widehat{O}(5B)}+\iint_{\mathbb{R}^{N+1}_{+}\backslash\widehat{O}(5B)}\right)\Big|p_{t}(x,y)-p_{t}(z,y)\Big|d|\mu|(y,t)d\omega(z)d\omega(x)\\
  &\quad=:II_{1}+II_{2}.
\end{align*}
For $II_{1}$, it follows from \eqref{eq:p-G-invariance} that
\begin{align*}
  II_{1}
  &\leq\frac{1}{\omega(B)}\iint_{\widehat{O}(5B)}\int_{\mathbb{R}^{N}}p_{t}(x,y)d\omega(x)d|\mu|(y,t)+\frac{1}{\omega(B)}\iint_{\widehat{O}(5B)}\int_{\mathbb{R}^{N}}p_{t}(z,y)d\omega(z)d|\mu|(y,t)\\
  &\lesssim\frac{|\mu|(\widehat{O}(5B))}{\omega(B)}\\
  &\lesssim\sum\limits_{\sigma\in{G}}\frac{|\mu|\left(\widehat{\sigma(5B)}\right)}{\omega(\sigma(B))}\\
  &\lesssim\|\mu\|_{\mathcal{C}},
\end{align*}
where we use the fact that since $d\omega$ is $G$-invariant,
 for any $\sigma\in{G}$,
\begin{equation*}
  \omega(B(x_{0}, r_{0}))\thicksim\omega(B(\sigma(x_{0}), r_{0})).
\end{equation*}

Consider $II_{2}$. For any $x,z\in{B}$, $y\in{\mathbb{R}^{N}\backslash O(5B)}$ and $t>0$, by \eqref{eq:poisson-Holder} (or \eqref{eq:poisson-smooth}),
\begin{equation*}
  \left|p_{t}(x,y)-p_{t}(z,y)\right|\lesssim\frac{\|x-z\|}{\omega(B(x, t+d(x,y)))}\frac{1}{t+\|x-y\|}\frac{t}{t+d(x,y)}.
\end{equation*}
Then
\begin{align*}
  II_{2}
  &\lesssim\sum\limits_{k=1}^{\infty}\frac{1}{\big(\omega(B)\big)^{2}}\int_{B}\int_{B}\iint_{\widehat{O}(5^{k+1}B)\backslash\widehat{O}(5^{k}B)}\frac{1}{\omega(B(x, t+d(x,y)))}\frac{\|x-z\|}{t+\|x-y\|}d|\mu|(y,t)d\omega(z)d\omega(x)\\
  &\lesssim\sum\limits_{k=1}^{\infty}\frac{1}{5^{k}}\frac{|\mu|\left(\widehat{O}(5^{k+1}B)\right)}{\omega(5^{k}B)}\\
  &\lesssim\|\mu\|_{\mathcal{C}}.
\end{align*}
Thus \eqref{ep:i-mu} holds and the proof of Theorem \ref{thm:balayage-measure} $\rm(i)$ is complete.
\end{proof}

To prove $\rm(ii)$ of Theorem \ref{thm:balayage-measure}, we also need the following lemma established in \cite{Uchiyama1980}.
\begin{lemma}\label{lem:measure-atoms}
Let $(\Omega, F, \nu)$ be a $\sigma$-finite measure space with no atoms. Let $\{F_{k}\}_{k\in{\mathbb{N}}}$ be a nondecreasing sequence of atomic sub-$\sigma$-fields of $F$. Let $A_{k}$ be the set of the atoms of $F_{k}$. Let $\{u_{k}(\cdot)\}_{k\in{\mathbb{N}}}$ be a sequence of nonnegative functions such that $\|\sum_{k\in{\mathbb{N}}}u_{k}\|_{L^{\infty}(\Omega,d\nu)}<1$. Then there exist mutually disjoint $F$-measurable sets $\{H_{k}\}_{k\in{\mathbb{N}}}$ such that
\begin{equation}\label{eq:atom-measure}
  \nu(H_{k}\cap Q)=\int_{Q}u_{k}(x)d\nu(x),
\end{equation}
for any $k\in{\mathbb{N}}$ and $Q\in{A_{k}}$.
\end{lemma}

\begin{proof}[Proof of $\rm(ii)$ in Theorem \ref{thm:balayage-measure}]
We assume that $\|f\|_{\ast,\kappa}$ is sufficiently small. Then by Theorem \ref{thm-BMO-onesetp}, there exist bounded functions $\{g_{i}\}_{i=0}^{\infty}$ and decreasing positive numbers $\{r_{i}\}_{i\in{\mathbb{N}}}$ such that 
\begin{equation*} 
 f=\sum\limits_{i=1}^{\infty}\int_{\R^N}\mathbf{p}_{r_{i}}(\cdot,y)g_{i}(y)\,d\omega(y)+g_{0}\qquad \mathrm{in}\quad\mathrm{BMO}(\mathbb{R}^{N}, d\omega),
\end{equation*}
and
\begin{equation*}
  \left\|\sum\limits_{i=1}^{\infty}|g_{i}|\right\|_{L^{\infty}(\mathbb{R}^{N}, d\omega)}<1,\quad \mathrm{and} \quad \|g_{0}\|_{L^{\infty}(\mathbb{R}^{N}, d\omega)}<1.
\end{equation*}
Let $\{\epsilon_{k}\}_{k\in{\mathbb{N}}}$ be positive numbers such that for any $k\in{\mathbb{N}}$, 
\begin{equation*}
 \frac{\epsilon_{k}}{\epsilon_{k+1}}\in{\mathbb{N}} \quad \mathrm{and} \quad 2^{-k}r_{k}>\sqrt{N}\epsilon_{k}.
\end{equation*}
Let $\{F_{k}\}_{k\in{\mathbb{N}}}$ be the atomic sub-$\sigma$-fields such that 
\begin{equation*}
A_{k}:=\left\{Q_{m}^{(k)}:\,\,m:=(m_{1},m_{2},\ldots,m_{N})\in{\mathbb{Z}^{N}}\right\},
\end{equation*}
where $k\in{\mathbb{N}}$ and 
\begin{equation*}
Q_{m}^{(k)}:=\prod_{i=1}^{N}\left[m_{i}\epsilon_{k},\, (m_{i}+1)\epsilon_{k}\right).
\end{equation*}
Let $k\in{\mathbb{N}}$. Then for any $Q,Q'\in{A_{k}}$, $Q\cap Q'=\emptyset$, $Q\neq Q'$ and $\mathbb{R}^{N}=\bigcup_{Q\in{A_{k}}}Q$. $A_{k}$ is as in Lemma \ref{lem:measure-atoms}. Set 
\begin{equation*}
u_{k}(x):=|g_{k}(x)|,
\end{equation*}
for $x\in{\mathbb{R}^{N}}$ and $k\in{\mathbb{N}}$. Then by Lemma \ref{lem:measure-atoms}, we get mutually disjoint measurable sets $\{H_{k}\}_{k\in{\mathbb{N}}}$ and \eqref{eq:atom-measure} for the measure $d\omega$.

Let $k\in{\mathbb{N}}$. For every $Q\in{A_{k}}$, we put
\begin{eqnarray*} 
f_{k}(x):=
\left\{\begin{array}{ll}
\frac{1}{\omega(H_{k}\cap Q)}\int_{Q}g_{k}(y)d\omega(y),\quad &x\in{H_{k}\cap Q}\,\,\,\text{and}\,\,\, \omega(H_{k}\cap Q)>0;\\
\,\\
0,\quad &\text{otherwise}.
\end{array}\right.
\end{eqnarray*}
It follows from \eqref{eq:atom-measure} that 
\begin{equation*}
  \|f_{k}\|_{L^{\infty}(\mathbb{R}^{N}, d\omega)}\leq1.
\end{equation*}
Moreover, for every $Q\in{A_{k}}$, we have
\begin{equation*}
\int_{Q}[g_{k}(y)-f_{k}(y)]d\omega(y)=0.
\end{equation*}
Hence
\begin{align*}
  &\left|\int_{\mathbb{R}^{N}}\mathbf{p}_{r_{k}}(x,y)(g_{k}(y)-f_{k}(y))d\omega(y)\right|\\
  &\quad\leq\sum\limits_{Q(x_{0},\ell(Q))\in{A_{k}}}\int_{Q(x_{0},\ell(Q))}|\mathbf{p}_{r_{k}}(x,y)-\mathbf{p}_{r_{k}}(x,x_{0})||g_{k}(y)-f_{k}(y)|d\omega(y)\\
  &\quad\lesssim\left(\|g_{k}\|_{L^{\infty}(\mathbb{R}^{N}, d\omega)}+\|f_{k}\|_{L^{\infty}(\mathbb{R}^{N}, d\omega)}\right) \sum\limits_{Q(x_{0},\ell(Q))\in{A_{k}}}\bigg(\int_{Q(x_{0},\ell(Q))}
  \frac{\|y-x_{0}\|}{\omega(B(x, r_{k}+d(x,y)))}\frac{1}{r_{k}+\|x-y\|}d\omega(y)\\
  &\quad\quad\quad\quad\quad\quad+\int_{Q(x_{0},\ell(Q))}
  \frac{\|y-x_{0}\|}{\omega(B(0, r_{k}+d(0,y)))}\frac{1}{r_{k}+\|y\|}d\omega(y)\bigg)\\
  &\quad\lesssim\frac{\sqrt{N}\epsilon_{k}}{r_{k}}\int_{\mathbb{R}^{N}}\left(\frac{1}{\omega(B(x, r_{k}+d(x,y)))}\frac{r_{k}}{r_{k}+\|x-y\|}+\frac{1}{\omega(B(0, r_{k}+d(0,y)))}\frac{r_{k}}{r_{k}+\|y\|}\right)d\omega(y)\\
  &\quad\lesssim2^{-k}.
\end{align*}
Here we have used $\ell(Q)=\epsilon_{k}$, for every $Q(x_{0},\ell(Q))\in{A_{k}}$. Hence we have 
\begin{equation*} 
 f=\sum\limits_{k=1}^{\infty}\int_{\mathbb{R}^{N}}\mathbf{p}_{r_{k}}(\cdot,y)f_{k}(y)d\omega(y)+\sum\limits_{k=1}^{\infty}\int_{\mathbb{R}^{N}}\mathbf{p}_{r_{k}}(\cdot,y)(g_{k}-f_{k})(y)d\omega(y)+g_{0},
\end{equation*}
in $\mathrm{BMO}(\mathbb{R}^{N}, d\omega)$, where 
\begin{equation*}
  \left\|\sum\limits_{k=1}^{\infty}\int_{\mathbb{R}^{N}}\mathbf{p}_{r_{k}}(\cdot,y)(g_{k}-f_{k})(y)d\omega(y)+g_{0}\right\|_{L^{\infty}(\mathbb{R}^{N}, d\omega)}\lesssim\sum\limits_{k=1}^{\infty}\frac{1}{2^{k}}+\|g_{0}\|_{L^{\infty}(\mathbb{R}^{N}, d\omega)}\lesssim1,
\end{equation*}
and $\|f_{k}\|_{L^{\infty}(\mathbb{R}^{N}, d\omega)}\leq1$, and $\supp f_{i}\cap\supp f_{j}=\emptyset$ for different $i,j\in{\mathbb{N}}$.

Let us define 
\begin{eqnarray*} 
b_{1}(y):=
\left\{\begin{array}{ll}
f_{k}(y),\quad &y\in{H_{k}},\,\,\,\text{for every}\,\,\, k\in{\mathbb{N}};\\
\,\\
0,\quad &y\in{\mathbb{R}^{N}\backslash\bigcup_{k\in{\mathbb{N}}}H_{k}},
\end{array}\right.
\quad
t(y):=
\left\{\begin{array}{ll}
r_{k},\quad &y\in{H_{k}},\,\,\,\text{for every}\,\,\, k\in{\mathbb{N}};\\
\,\\
1,\quad &y\in{\mathbb{R}^{N}\backslash\bigcup_{k\in{\mathbb{N}}}H_{k}},
\end{array}\right.
\end{eqnarray*}
and 
\begin{equation*}
  b_{0}:=\sum\limits_{k=1}^{\infty}\int_{\mathbb{R}^{N}}\mathbf{p}_{r_{k}}(\cdot,y)(g_{k}-f_{k})(y)d\omega(y)+g_{0}.
\end{equation*}
Then we have
\begin{align}\label{eq:convolution-decomposition}
 f(x)\notag
 &=\sum\limits_{k=1}^{\infty}\int_{\mathbb{R}^{N}}\mathbf{p}_{r_{k}}(x,y)f_{k}(y)d\omega(y)+b_{0}(x)=\sum\limits_{k=1}^{\infty}\int_{H_{k}}\mathbf{p}_{t(y)}(x,y)b_{1}(y)d\omega(y)+b_{0}(x)\\ 
 &=\int_{\mathbb{R}^{N}}\mathbf{p}_{t(y)}(x,y)b_{1}(y)d\omega(y)+b_{0}(x),
\end{align}
in $\mathrm{BMO}(\mathbb{R}^{N}, d\omega)$.

Let us define
\begin{equation*}
  d\nu(y,t):=b_{1}(y)\omega(y)ds(y,t),
\end{equation*}
where $ds$ denotes $N$-dimensional surface measure on the graph $\{(y,t(y)):\,y\in{\mathbb{R}^{N}}\}$, or equivalently,
\begin{equation*}
d\nu(y,t):=b_{1}(y)\delta_{t=t(y)}d\omega(y),
\end{equation*}
where $\delta_{t=t(y)}$ is the Dirac measure concentrated at $t=t(y)$. Then we obtain
\begin{equation*}
f(x)=\iint_{\mathbb{R}^{N+1}_{+}}p_{t}(x,y)-p_{t}(0,y)\cdot \mathbf{1}_{\{t>1\}}(y,t)d\nu(y,t)+b_{0}(x)\qquad \mathrm{in}\quad\mathrm{BMO}(\mathbb{R}^{N}, d\omega).
\end{equation*}
It is straightforward to verify that for any ball $B:=B(x,r)\subset\mathbb{R}^{N}$,
\begin{equation*}
 |\nu|(B\times(0,r))=\int_{B}|b_{1}(y)|\mathbf{1}_{\{B\times(0,r)\}}(y, t(y))d\omega(y)\leq\|b_{1}\|_{L^{\infty}(\mathbb{R}^{N}, d\omega)}\omega(B).
\end{equation*}
Therefore, we obtain $\rm(ii)$ in Theorem \ref{thm:balayage-measure} and the proof is finished.
\end{proof}

\begin{remark}\label{rem-bmo-descom}
If we omit $h(0)$ from \eqref{eq:f-addh0}, then by arguments analogous to those used in the proof of Theorem~\ref{thm-BMO-onesetp} and in deriving \eqref{eq:convolution-decomposition}, we also obtain decompositions analogous to those introduced by Carleson and Uchiyama \cite{CarlesonLennart1976,Uchiyama1980}. More precisely, for every
$f\in \mathrm{BMO}(\mathbb{R}^{N},d\omega)$ with compact support, there exist bounded functions
$\{g_i\}_{i=0}^{\infty}$ and functions $t_i(\cdot):\,\mathbb{R}^{N}\rightarrow(0,\infty)$, $i\in\mathbb{N}$ such that
\begin{equation*}
f
=
\sum_{i=1}^{\infty}
\int_{\mathbb{R}^{N}}
p_{t_i(y)}(\,\cdot\,,y)g_i(y)\,d\omega(y)
+g_0
\qquad\text{in }\mathrm{BMO}(\mathbb{R}^{N},d\omega).
\end{equation*}
Moreover, there exist bounded functions $b_0,b_1$ and a bounded function $t(\cdot):\,\mathbb{R}^{N}\rightarrow(0,\infty)$ such that
\begin{equation*} 
f
=
\int_{\mathbb{R}^{N}}
p_{t(y)}(\cdot,y)b_1(y)\,d\omega(y)
+b_0
\qquad\text{in }\mathrm{BMO}(\mathbb{R}^{N},d\omega).
\end{equation*}
\end{remark}

\subsection{Proof of Theorem \ref{thm-Hardy-balayage}}\label{sec:thm1.3}
\hskip\parindent

This subsection is devoted to presenting the proof of Theorem \ref{thm-Hardy-balayage}. To begin with, we establish the following lemma.
\begin{lemma}\label{l-BMO-H1}
Let $f\in{H^{1}(\mathbb{R}^{N}, d\omega)}$ and $b\in{\mathrm{BMO}(\mathbb{R}^{N}, d\omega)}$. Then $\pv\int_{\mathbb{R}^{N}} f(y)b(y)d\omega(y)$ exists and
\begin{equation*}
\left|
\pv\int_{\mathbb{R}^{N}} f(y)b(y)d\omega(y)
\right|
\lesssim_{\mathbf{N}}\|f\|_{H^{1}(\mathbb{R}^{N}, d\omega)}\|b\|_{\ast,\kappa}.
\end{equation*}
\end{lemma}
\begin{proof}
Take an atomic decomposition of $f\in H^1(\mathbb{R}^{N}, d\omega)$: $f=\sum_{j=1}^{\infty}\alpha_j a_j$ and $\sum_{j=1}^{\infty}|\alpha_j|<\infty$,
where for every $j\in{\mathbb{N}}$, there exists a ball $B_{j}\subset\mathbb{R}^{N}$ such that $\supp a_j\subset B_{j}$ and
\begin{equation*}
\|a_j\|_{L^{\infty}(\mathbb{R}^{N}, d\omega)}\leq \frac{1}{\omega(B_j)}\quad \mathrm{and}\quad \int_{\mathbb{R}^{N}}a_j(x)d\omega(x)=0.
\end{equation*}

We first claim that 
\begin{equation}\label{eq:fbN}
\lim_{K\to\infty}\int_{\mathbb{R}^{N}} f(y)b_K(y)d\omega(y)
=
\sum_{j=1}^{\infty}\alpha_j\int_{\mathbb{R}^{N}} a_j(y)b(y)d\omega(y),
\end{equation}
where $b_{K}$ is as in \eqref{eq:bN-notation}.
Indeed, for every $j\in{\mathbb{N}}$ and $K>0$, we have
\begin{equation}\label{eq:f-atom-decom}
\left|\int_{\mathbb{R}^{N}} a_j(y)b(y)d\omega(y)\right|\leq \|b\|_{\ast,\kappa}\qquad \text{and }\qquad \left|\int_{\mathbb{R}^{N}} a_j(y)b_{K}(y)d\omega(y)\right|\leq\frac{9}{4} \|b\|_{\ast,\kappa},
\end{equation}
where we have used the standard estimate $\|b_{K}\|_{\ast,\kappa}\leq\frac{9}{4} \|b\|_{\ast,\kappa}$ for every $K>0$.
Since $|b_{K}|\leq K$, \eqref{eq:f-atom-decom} implies
\[
\int_{\mathbb{R}^{N}} f(y)b_K(y)d\omega(y)
=
\sum_{j=1}^{\infty}\alpha_j\int_{\mathbb{R}^{N}} a_j(y)b_{K}(y)d\omega(y).
\]

Notice that 
\[
\lim\limits_{K\rightarrow\infty}b_K(y)=b(y)
\qquad\text{for\,\,a.e.}\,\,y\in{\mathbb{R}^{N}},
\qquad
|b_K|\le |b|.
\]
Since $b\in{L_{\mathrm{loc}}^{1}(\mathbb{R}^{N}, d\omega})$, the dominated convergence theorem yields
\[
\lim\limits_{K\rightarrow\infty}\int_{\mathbb{R}^{N}} a_j(y)b_{K}(y)d\omega(y)=\int_{\mathbb{R}^{N}} a_j(y)b(y)d\omega(y).
\]
Moreover, \eqref{eq:f-atom-decom} yields that
\[
\left|\alpha_j\int_{\mathbb{R}^{N}} a_j(y)b_{K}(y)d\omega(y)\right|
\le
\frac{9}{4} \|b\|_{\ast,\kappa}|\alpha_j|,\qquad \sum_{j=1}^{\infty}\|b\|_{\ast,\kappa}|\alpha_j|<\infty.
\]
Thus, the dominated convergence theorem for series allows us to interchange the limit and the sum in
\begin{equation*}
\lim_{K\to\infty}
\sum_{j=1}^{\infty}\alpha_j
\int_{\mathbb{R}^{N}} a_j(y)b_K(y)\,d\omega(y),
\end{equation*}
yielding \eqref{eq:fbN}. Therefore, we conclude that Lemma \ref{l-BMO-H1} holds.
\end{proof}

\begin{proof}[Proof of Theorem \ref{thm-Hardy-balayage}]
Set 
\begin{equation*}
\mathcal{N}(f):=\sup\limits_{\substack{\mu\in{\Re(\mathbb{R}^{N+1}_{+})},\\\|\mu\|_{\mathcal{C}}=1}}\left|\operatorname{p.v.}\int_{\mathbb{R}^{N}}f(y)\cdot \widetilde{S}_{\mu,\mathcal{P}}(y)d\omega(y)\right|
\end{equation*}
and 
\[
\mathcal{N}^\sharp(f):=
\sup\limits_{\substack{\mu\in{\Re(\mathbb{R}^{N+1}_{+})},\\\|\mu\|_{\mathcal{C}}=1,\ c\in \mathbb{R}}}\left|\operatorname{p.v.}\int_{\mathbb{R}^{N}}f(y)\cdot \left(\widetilde{S}_{\mu,\mathcal{P}}(y)+c\right)d\omega(y)\right|.
\]

To prove $(\rm{i})\Longrightarrow(\rm{ii})$, assume $f\in{H^{1}(\mathbb{R}^{N}, d\omega)}$. Then for any $c\in\mathbb{R}$,
\begin{equation*}
\int_{\mathbb R^{N}} cf(y)\,d\omega(y)=0.
\end{equation*}
By $\rm(i)$ in Theorem \ref{thm:balayage-measure}, we have that for any $\mu\in{\Re(\mathbb{R}^{N+1}_{+})}$, $\widetilde{S}_{\mu,\mathcal{P}}\in{\mathrm{BMO}(\mathbb{R}^{N}, d\omega)}$. Applying Lemma \ref{l-BMO-H1} and \eqref{ep:i-mu}, we obtain that for every $c\in\mathbb{R}$,
\begin{equation*}
\left|\operatorname{p.v.}\int_{\mathbb{R}^{N}}f(y)\cdot \left(\widetilde{S}_{\mu,\mathcal{P}}(y)+c\right)d\omega(y)\right|=\left|\operatorname{p.v.}\int_{\mathbb{R}^{N}}f(y)\cdot \widetilde{S}_{\mu,\mathcal{P}}(y)d\omega(y)\right|\lesssim\|f\|_{H^{1}(\mathbb{R}^{N}, d\omega)}\|\mu\|_{\mathcal{C}}.
\end{equation*}
Taking the supremum over all admissible $\mu\in{\Re(\mathbb{R}^{N+1}_{+})}$ with $\|\mu\|_{\mathcal{C}}=1$ and all $c\in\mathbb{R}$, we get
\begin{equation*}
\mathcal{N}^\sharp(f)\lesssim\|f\|_{H^{1}(\mathbb{R}^{N}, d\omega)}.
\end{equation*}

Consider $(\rm{ii})\Longrightarrow(\rm{i})$. Firstly, we establish the following fact. For $f\in L^1(\mathbb{R}^{N}, d\omega)$, $c\in{\mathbb{R}}$ and $b\in{\mathrm{BMO}(\mathbb{R}^{N}, d\omega)}$, it follows from 
\[
|(b(y)+c)_K-b_K(y)|\leq |c|, 
\]
for each $K>0$ and $y\in{\mathbb{R}^{N}}$, and the dominated convergence that
\[
\lim_{K\to\infty}
\int_{\mathbb R^{N}} f(y)\bigl[(b(y)+c)_K-b_K(y)\bigr]\,d\omega(y)
=
c\int_{\mathbb R^{N}} f(y)\,d\omega(y).
\]

We now claim that for every $f\in L^{1}(\mathbb R^{N},d\omega)$, $\mathcal{N}^\sharp(f)<\infty$ if and only if
\begin{equation}\label{eq:Nf-estimate}
\mathcal{N}(f)<\infty
\quad\text{and}\quad
\int_{\mathbb R^{N}} f(y)\,d\omega(y)=0.
\end{equation}
Indeed, if $\mathcal{N}^\sharp(f)<\infty$, then by taking $c=0$ we obtain
\[
\mathcal{N}(f)\le \mathcal{N}^\sharp(f)<\infty.
\]
Also, by taking $d\mu(y,t):=1\cdot\delta_{t=1}d\omega(y)$, we get $\|\mu\|_{\mathcal{C}}=1$. Then $\widetilde{S}_{\mu,\mathcal{P}}=S_{\mu,\mathcal{P}}=1$ and 
\[
\sup_{c\in \mathbb{R}}\left|
\int_{\mathbb R^{N}} f(y)\,(c+1)\,d\omega(y)
\right|
\le \mathcal{N}^\sharp(f)<\infty.
\]
This together with $\sup_{c\in\mathbb{R}}|c|=\infty$ further implies that $\int_{\mathbb R^{N}} f(y)\,d\omega(y)=0$. Conversely, it is obvious that $\mathcal{N}^\sharp(f)<\infty$ follows from \eqref{eq:Nf-estimate}. Then the claim holds.

Assume $g\in{C_{c}(\mathbb{R}^{N})}$, the space of all continuous functions on $\mathbb{R}^{N}$ with compact support. From $\rm(ii)$ in Theorem \ref{thm:balayage-measure}, after separating off an additive constant, there exist a bounded function $g_{0}$, a constant $c_{0}$ and a Carleson measure $\mu$ such that 
\begin{equation*} 
g(x)=g_{0}(x)+ \widetilde{S}_{\mu,\mathcal{P}}(x)+c_{0},\qquad \text{for a.e.}\,\,\,x\in{\mathbb{R}^{N}},
\end{equation*}
with $\|g_{0}\|_{L^{\infty}(\mathbb{R}^{N}, d\omega)}+\|\mu\|_{\mathcal{C}}\lesssim\|g\|_{\ast,\kappa}$.

Since 
\begin{equation*}
\widetilde{S}_{\mu,\mathcal{P}}=g-g_{0}-c_{0}\in L^\infty(\mathbb{R}^{N}, d\omega),
\end{equation*}
we have that for $f\in L^1(\mathbb{R}^{N}, d\omega)$,
\[
\left|f\widetilde{S}_{\mu,\mathcal{P}}\right|
\le
\left\|\widetilde{S}_{\mu,\mathcal{P}}\right\|_{L^\infty(\mathbb{R}^{N}, d\omega)}|f|\in L^1(\mathbb{R}^{N}, d\omega).
\]
Then the dominated convergence therefore gives
\begin{equation*}
\int_{\mathbb{R}^{N}} f(y)\widetilde{S}_{\mu,\mathcal{P}}(y)d\omega(y)
=
\lim_{K\to\infty}
\int_{\mathbb{R}^{N}} f(y)\left[\widetilde{S}_{\mu,\mathcal{P}}\right]_{K}(y)d\omega(y),
\end{equation*}
where $\left[\widetilde{S}_{\mu,\mathcal{P}}\right]_{K}:=T_K\circ \widetilde{S}_{\mu,\mathcal{P}}$. Hence, if $\|\mu\|_{\mathcal{C}}\neq0$, then by \eqref{eq:Nf-estimate} and the principal-value definition, 
\begin{align}\label{eq:linear-continuous}
  \left|\mathcal{L}_{f}(g)\right|
  &:=\left|\int_{\mathbb{R}^{N}}f(x)g(x)d\omega(x)\right|\leq\left|\int_{\mathbb{R}^{N}} f(y)\widetilde{S}_{\mu,\mathcal{P}}(y)d\omega(y)\right|+\left|\int_{\mathbb{R}^{N}}f(x)g_{0}(x)d\omega(x)\right|\\\notag
  &\leq\|\mu\|_{\mathcal{C}}\left|\operatorname{p.v.}\int_{\mathbb{R}^{N}}f(y) \iint_{\mathbb{R}^{N+1}_{+}}\mathbf{p}_{t}(y,z)\frac{d\mu(z,t)}{\|\mu\|_{\mathcal{C}}}d\omega(y)\right|+\|g_{0}\|_{L^{\infty}(\mathbb{R}^{N}, d\omega)}\|f\|_{L^{1}(\mathbb{R}^{N}, d\omega)}\\\notag
  &\lesssim\left\|g\right\|_{\ast,\kappa}\left(\|f\|_{L^{1}(\mathbb{R}^{N}, d\omega)}+\mathcal{N}(f)\right).\notag
\end{align}
Moreover, when $\|\mu\|_{\mathcal{C}}=0$, it is obvious that \eqref{eq:linear-continuous} also holds.
This estimate shows that $\mathcal{L}_{f}$ is continuous on $C_{c}(\mathbb R^{N})$ with respect to the $\mathrm{BMO}$ seminorm $\|\cdot\|_{\ast,\kappa}$.
By a standard density argument, $\mathcal{L}_{f}$ extends uniquely to a bounded linear functional on $\mathrm{CMO}(\mathbb{R}^{N}, d\omega)$, the $\mathrm{BMO}(\mathbb{R}^{N}, d\omega)$-closure of $C_{c}(\mathbb{R}^{N})$.

Since $H^{1}(\mathbb{R}^{N}, d\omega)$ is the dual space of $\mathrm{CMO}(\mathbb{R}^{N}, d\omega)$; see \cite[Theorem 5.5]{GuoLiWick2026}, there exists $h\in H^{1}(\mathbb R^{N},d\omega)$ such that
\[
\mathcal{L}_{f}(g)=\int_{\mathbb R^{N}} h(x)\,g(x)\,d\omega(x),
\qquad g\in C_{c}(\mathbb R^{N}).
\]
Hence
\[
\int_{\mathbb R^{N}} (f-h)(x)\,g(x)\,d\omega(x)=0,
\qquad g\in C_{c}(\mathbb R^{N}).
\]
Therefore $f(x)=h(x)$, a.e. $x\in{\mathbb{R}^{N}}$, and thus $f\in H^{1}(\mathbb R^{N},d\omega)$. Furthermore, by \eqref{eq:linear-continuous} 
\begin{equation*}
\|f\|_{H^{1}(\mathbb{R}^{N}, d\omega)}\lesssim\|f\|_{L^{1}(\mathbb{R}^{N}, d\omega)}+\mathcal{N}(f).
\end{equation*}
So the proof of Theorem \ref{thm-Hardy-balayage} is finished.
\end{proof}

\section{Generalized balayage associated to the Dunkl--heat semigroup}\label{sec:Dunkl-heat}

The purpose of this section is to study the generalized balayage of Carleson measures induced by the Dunkl--heat semigroup. Let $\{\mathcal{H}_{t}\}_{t>0}$ be the heat semigroup associated with $\Delta_D$ defined by 
\begin{equation*}
\mathcal{H}_{t}f(x):=e^{t\Delta_D}f(x):=\int_{\mathbb{R}^{N}}h_{t}(x,y)f(y)d\omega(y),\quad x\in{\mathbb{R}^{N}},\,t>0,
\end{equation*}
where
\begin{equation*}
h_{t}(x,y):= \frac{(2t)^{-\mathbf{N}/2}}{c_{\kappa}}e^{-\frac{\|x\|^{2}+\|y\|^{2}}{4t}}\int_{\mathbb{R}^{N}}e^{\langle\xi, \frac{y}{2t}\rangle}d\mu_{x}(\xi).
\end{equation*}
We define the balayage associated with $\{\mathcal{H}_{t}\}_{t>0}$ by
$$
  \widetilde{S}_{\nu,\mathcal{H}}(x):=\iint_{\mathbb{R}^{N+1}_{+}}h_{t^{2}}(x,y)-h_{t^{2}}(0,y)\cdot \mathbf{1}_{\{t>1\}}(y,t)d\nu(y,t), 
$$
where $\nu\in{\Re(\mathbb{R}^{N+1}_{+})}$ and $x\in{\mathbb{R}^{N}}$.

\begin{theorem}\label{thm:balayage-measure-heat}
For every $f\in{\mathrm{BMO}(\mathbb{R}^{N}, d\omega)}$ with compact support in $\mathbb{R}^{N}$, there exist $b\in{L^{\infty}(\mathbb{R}^{N}, d\omega)}$ and a Carleson measure $\nu\in{\Re(\mathbb{R}^{N+1}_{+})}$ such that
\begin{equation*} 
f=b+\widetilde{S}_{\nu,\mathcal{H}},\qquad \mathrm{in}\quad\mathrm{BMO}(\mathbb{R}^{N}, d\omega),
\end{equation*}
where
\begin{equation*}
\|b\|_{L^{\infty}(\mathbb{R}^{N}, d\omega)}+\|\nu\|_{\mathcal{C}}\lesssim\|f\|_{\ast,\kappa}.
\end{equation*}
Conversely, if $\nu\in{\Re(\mathbb{R}^{N+1}_{+})}$, then $\widetilde{S}_{\nu,\mathcal{H}}\in{\mathrm{BMO}(\mathbb{R}^{N}, d\omega)}$ with $\left\|\widetilde{S}_{\nu,\mathcal{H}}\right\|_{\ast,\kappa}\lesssim\|\nu\|_{\mathcal{C}}$.
Here, the implicit constants are independent of $f$ and $\nu$.
\end{theorem}

\begin{proof} 
For every $f\in{\mathrm{BMO}(\mathbb{R}^{N}, d\omega)}$ with compact support on $\mathbb{R}^{N}$, from $\rm(ii)$ in Theorem \ref{thm:balayage-measure}, there exist $g\in{L^{\infty}(\mathbb{R}^{N}, d\omega)}$ and a Carleson measure $\mu\in{\Re(\mathbb{R}^{N+1}_{+})}$ such that
\begin{equation*} 
f=g+\widetilde{S}_{\mu,\mathcal{P}},\qquad \mathrm{in}\quad\mathrm{BMO}(\mathbb{R}^{N}, d\omega),
\end{equation*}
with
\begin{equation*}
\|g\|_{L^{\infty}(\mathbb{R}^{N}, d\omega)}+\|\mu\|_{\mathcal{C}}\lesssim\|f\|_{\ast,\kappa}.
\end{equation*}
From this, and the subordination formula
\begin{equation*}
  p_{t}(x,y)=\frac{t}{\sqrt{\pi}}\int_{0}^{\infty}\frac{e^{-\frac{t^{2}}{4s^{2}}}}{s^{2}}h_{s^{2}}(x,y)ds,\qquad x,y\in{\mathbb{R}^{N}},t>0;
\end{equation*}
see \cite[(23)]{JiuLi2023}, we rewrite
\begin{align*}
  \widetilde{S}_{\mu,\mathcal{P}}
  &=\iint_{\mathbb{R}^{N+1}_{+}}p_{t}(x,y)-p_{t}(0,y)\cdot \mathbf{1}_{\{t>1\}}(y,t)d\mu(y,t)\\
  &=\iint_{\mathbb{R}^{N+1}_{+}}\left[h_{s^{2}}(x,y)-h_{s^{2}}(0,y)\cdot \mathbf{1}_{\{s>1\}}(y,s)\right]\left(\frac{1}{\sqrt{\pi}}\int_{0<t<\infty}t\frac{e^{-\frac{t^{2}}{4s^{2}}}}{s^{2}}d\mu(y,t)ds\right)\\
  &\quad+\int_{0}^{\infty}\int_{\mathbb{R}^{N}}h_{s^{2}}(0,y)\cdot \mathbf{1}_{\{s>1\}}(y,s)\frac{1}{\sqrt{\pi}}\int_{0<t<\infty}t\frac{e^{-\frac{t^{2}}{4s^{2}}}}{s^{2}}\mathbf{1}_{\{t\leq1\}}(y,t)d\mu(y,t)ds\\
  &\quad\quad-\int_{0}^{\infty}\int_{\mathbb{R}^{N}}h_{s^{2}}(0,y)\cdot \mathbf{1}_{\{s\leq1\}}(y,s)\frac{1}{\sqrt{\pi}}\int_{0<t<\infty}t\frac{e^{-\frac{t^{2}}{4s^{2}}}}{s^{2}}\mathbf{1}_{\{t>1\}}(y,t)d\mu(y,t)ds.
\end{align*}
It reduces to show that
\begin{equation*}
  d\nu(y,s):=\frac{1}{\sqrt{\pi}}\int_{0<t<\infty}t\frac{e^{-\frac{t^{2}}{4s^{2}}}}{s^{2}}d\mu(y,t)ds
\end{equation*}
belongs to $\Re(\mathbb{R}^{N+1}_{+})$, and 
\begin{equation*}
  C_{1}:=\frac{1}{\sqrt{\pi}}\int_{0}^{\infty}\int_{\mathbb{R}^{N}}h_{s^{2}}(0,y)\cdot \mathbf{1}_{\{s>1\}}(y,s)\int_{0<t<\infty}t\frac{e^{-\frac{t^{2}}{4s^{2}}}}{s^{2}}\mathbf{1}_{\{t\leq1\}}(y,t)d\mu(y,t)ds
\end{equation*}
and 
\begin{equation*}
  C_{2}:=\frac{1}{\sqrt{\pi}}\int_{0}^{\infty}\int_{\mathbb{R}^{N}}h_{s^{2}}(0,y)\cdot \mathbf{1}_{\{s\leq1\}}(y,s)\int_{0<t<\infty}t\frac{e^{-\frac{t^{2}}{4s^{2}}}}{s^{2}}\mathbf{1}_{\{t>1\}}(y,t)d\mu(y,t)ds
\end{equation*}
are bounded constants controlled by $\|\mu\|_{\mathcal{C}}$.

Indeed, for each ball $B:=B(x_{0},r)\subset\mathbb{R}^{N}$,
\begin{align*}
  &\int_{0}^{r}\int_{B}\int_{0<t<\infty}t\frac{e^{-\frac{t^{2}}{4s^{2}}}}{s^{2}}d|\mu|(y,t)ds\\
  &\quad=\int_{0}^{r}\int_{B}\int_{0<t<r}t\frac{e^{-\frac{t^{2}}{4s^{2}}}}{s^{2}}d|\mu|(y,t)ds+\sum_{k=1}^{\infty}\int_{0}^{r}\int_{B}\int_{2^{k-1}r\leq t<2^{k}r}t\frac{e^{-\frac{t^{2}}{4s^{2}}}}{s^{2}}d|\mu|(y,t)ds\\
  &\quad\leq|\mu|(\widehat{B})\int_{0}^{\infty}\frac{e^{-\frac{1}{4s^{2}}}}{s^{2}}ds+\sum_{k=1}^{\infty}|\mu|(\widehat{2^{k}B})\int_{0}^{2^{1-k}}\frac{e^{-\frac{1}{4s^{2}}}}{s^{2}}ds\\
  &\quad\lesssim|\mu|(\widehat{B})+\sum_{k=1}^{\infty}2^{-k(\mathbf{N}+1)}|\mu|(\widehat{2^{k}B})\\
  &\quad\lesssim\|\mu\|_{\mathcal{C}}\omega(B),
\end{align*}
where in the penultimate inequality, we use the following estimate 
\begin{equation*}
\int_{0}^{2^{1-k}}\frac{e^{-\frac{1}{4s^{2}}}}{s^{2}}ds\lesssim_{\mathbf{N}}2^{-k(\mathbf{N}+1)}.
\end{equation*}
It shows $\nu\in{\Re(\mathbb{R}^{N+1}_{+})}$ with $\|\nu\|_{\mathcal{C}}\lesssim\|f\|_{\ast,\kappa}$. 

Next, consider $C_{1}$ and $C_{2}$. According to \cite[Proposition 5.1]{Margit1999}, when $x=0$, we have
\begin{equation*}
\int_{\mathbb{R}^{N}}e^{\langle\xi, \frac{y}{2t}\rangle}d\mu_{x=0}(\xi)=1,
\end{equation*}
for every $y\in\mathbb{R}^{N}$ and $t>0$. Thus
\begin{equation*}
h_{s^{2}}(0,y)= \frac{(\sqrt{2}s)^{-\mathbf{N}}}{c_{\kappa}}e^{-\frac{\|y\|^{2}}{4s^{2}}}.
\end{equation*}
Then
\begin{align*}
  |C_{1}|
  &\lesssim\int_{1<s<\infty}\iint_{\mathbb{R}^{N}\times(0,1]}e^{-\frac{\|y\|^{2}}{4s^{2}}}\frac{1}{s^{\mathbf{N}+2}}d|\mu|(y,t)ds\\
  &\lesssim\iint_{\mathbb{R}^{N}\times(0,1]}\frac{1}{(1+\|y\|)^{\mathbf{N}+1}}d|\mu|(y,t)\\
  &\lesssim\left(\iint_{B(0,2)\times(0,1]}+\sum_{k=1}^{\infty}\iint_{[B(0,2^{k+1})\backslash B(0,2^{k})]\times(0,1]}\right)\frac{1}{(1+\|y\|)^{\mathbf{N}+1}}d|\mu|(y,t)\\
  &\lesssim\|\mu\|_{\mathcal{C}},
\end{align*}
where in the second inequality, we use the fact that for any $y\in{\mathbb{R}^{N}}$,
\begin{equation*}
\int_{1}^{\infty}e^{-\frac{\|y\|^{2}}{4s^{2}}}\frac{1}{s^{\mathbf{N}+2}}ds\lesssim_{\mathbf{N}}\frac{1}{(1+\|y\|)^{\mathbf{N}+1}}.
\end{equation*}

Using the estimate that for any $y\in{\mathbb{R}^{N}}$ and $t>0$ with $1\leq t^{2}+\|y\|^{2}$,
\begin{equation*}
  \int_{0}^{1}\frac{te^{-\frac{\|y\|^{2}+t^{2}}{4s^{2}}}}{s^{\mathbf{N}+2}}ds\lesssim_{\mathbf{N}}e^{-(\|y\|^{2}+t^{2})/4},
\end{equation*}
and the inclusion 
\begin{equation*}
  \left\{(y,t)\in\mathbb{R}^{N+1}_{+}:\,\,2^{k}\leq\sqrt{\|y\|^{2}+t^{2}}<2^{k+1}\right\}\subset B(0,2^{k+1})\times(0,2^{k+1}),
\end{equation*}
for every $k\in{\{0,1,2,\ldots\}}$, we obtain
\begin{align*}
  |C_{2}|
  &\lesssim \int_{1<t<\infty}\int_{\mathbb{R}^{N}}\int_{0}^{1}\frac{te^{-\frac{\|y\|^{2}+t^{2}}{4s^{2}}}}{s^{\mathbf{N}+2}}dsd|\mu|(y,t)\\
  &\lesssim\int_{1}^{\infty}\int_{\mathbb{R}^{N}}e^{-(\|y\|^{2}+t^{2})/4}d|\mu|(y,t)\\
  &\lesssim\sum_{k=0}^{\infty}\iint_{2^{k}\leq\sqrt{\|y\|^{2}+t^{2}}<2^{k+1}}e^{-(\|y\|^{2}+t^{2})/4}d|\mu|(y,t)\\
  &\lesssim\|\mu\|_{\mathcal{C}}\sum_{k=0}^{\infty}2^{\mathbf{N}k}e^{-4^{k-1}}\\
  &\lesssim\|\mu\|_{\mathcal{C}}.
\end{align*}
Let $b:=g+C_{1}-C_{2}$. Then we have $f=b+\widetilde{S}_{\nu,\mathcal{H}}$, in $\mathrm{BMO}(\mathbb{R}^{N}, d\omega)$.

On the other hand, the heat semigroup $\{\mathcal{H}_{t}\}_{t>0}$ has the following properties:

$\rm(i)$ For all $x,y\in\R^N$ and $t>0$,
\begin{equation*} 
h_t(x,y)=h_t(y,x)>0, \quad \text{and} \quad\int_{\mathbb{R}^{N}}h_{t}(x,y)d\omega(y)=1.
\end{equation*}

$\rm(ii)$ There are constants $C_{10},C_{11}>0$ such that
\begin{equation*} 
\left\|\nabla_{x}h_{t}(x,y)\right\|\leq C_{10}\frac{1}{\sqrt{t}}\left(1+\frac{\|x-y\|}{\sqrt{t}}\right)^{-2}\frac{1}{V(x,y, \sqrt{t})}e^{-C_{11}d(x,y)^{2}/t},
\end{equation*}
for every $x,y\in\R^N$ and $t>0$; see \cite{JacekHejna2020}.

Using the above properties of $\{\mathcal{H}_{t}\}_{t>0}$ and arguing as in the proof of Theorem \ref{thm:balayage-measure} $\rm(i)$, we also conclude that $\widetilde{S}_{\nu,\mathcal{H}}\in{\mathrm{BMO}(\mathbb{R}^{N}, d\omega)}$ for every $\nu\in{\Re(\mathbb{R}^{N+1}_{+})}$. Therefore, the proof of this theorem is complete.
\end{proof}

\begin{remark}\label{rem-Hardy-heat}
By Theorem \ref{thm:balayage-measure-heat} and an argument similar to that used in the proof of Theorem \ref{thm-Hardy-balayage}, the conclusion of Theorem \ref{thm-Hardy-balayage} remains valid when $\widetilde{S}_{\mu,\mathcal{P}}$ is replaced by $\widetilde{S}_{\nu,\mathcal{H}}$.
\end{remark}

\bigskip
\noindent\textbf{Acknowledgements}\  \  
The author would like to thank Professor Ji Li for valuable suggestions. 
This work is supported by the NNSF of China [12601176], the NSFYSIP of Fujian Provincial [2026J00831 4] and the start-up fund of XMUT [YKJ25060R]. 

\vspace{0.3cm}
\noindent\textbf{Declarations}

\vspace{0.3cm}
\noindent\textbf{Data availability}\  \  
No datasets were generated or analyzed during the current study. 

\vspace{0.3cm}
\noindent\textbf{Conflict of Interest}\ \
No conflict of interest was reported by the author.


\end{document}